\pdfoutput=1
\RequirePackage{ifpdf}
\ifpdf 
\documentclass[pdftex]{sigma}
\else
\documentclass{sigma}
\fi

\numberwithin{equation}{section}

\newtheorem{Theorem}{Theorem}[section]
\newtheorem{Corollary}[Theorem]{Corollary}
\newtheorem{Lemma}[Theorem]{Lemma}
\newtheorem{Proposition}[Theorem]{Proposition}
 { \theoremstyle{definition}
\newtheorem{Definition}[Theorem]{Definition}
\newtheorem{Remark}[Theorem]{Remark} }

\newcommand{\nbige}{\mathcal{E}}
\newcommand{\nbigf}{\mathcal{F}}

\newcommand{\nbigl}{\mathcal{L}}
\newcommand{\nbigm}{\mathcal{M}}

\newcommand{\nbigo}{\mathcal{O}}

\newcommand{\nbigu}{\mathcal{U}}
\newcommand{\nbigv}{\mathcal{V}}

\newcommand{\seisuu}{{\mathbb Z}}

\newcommand{\cnum}{{\mathbb C}}
\newcommand{\real}{{\mathbb R}}

\newcommand{\gbigl}{\mathfrak L}

\newcommand{\gbigt}{\mathfrak T}

\newcommand{\gminia}{\mathfrak a}

\newcommand{\gminit}{\mathfrak t}

\newcommand{\gminiv}{\mathfrak v}

\newcommand{\vecx}{{\boldsymbol x}}

\newcommand{\vectau}{{\boldsymbol \tau}}

\newcommand{\lrarr}{\longrightarrow}

\def\End{\mathop{\rm End}\nolimits}

\def\Image{\mathop{\rm Im}\nolimits}

\def\Re{\mathop{\rm Re}\nolimits}

\def\rank{\mathop{\rm rank}\nolimits}

\def\length{\mathop{\rm length}\nolimits}

\def\Tr{\mathop{\rm Tr}\nolimits}
\def\vol{\mathop{\rm vol}\nolimits}
\def\dvol{\mathop{\rm dvol}\nolimits}

\def\id{\mathop{\rm id}\nolimits}

\newcommand{\del}{\partial}
\newcommand{\delbar}{\overline{\del}}

\newcommand{\barz}{\overline{z}}
\newcommand{\zbar}{\barz}

\newcommand{\baralpha}{\overline{\alpha}}
\newcommand{\alphabar}{\baralpha}

\newcommand{\vecnbigl}{{\boldsymbol{\mathcal L}}}

\newcommand{\closedopen}[2]{[#1,#2[}

\newcommand{\Etilde}{\widetilde{E}}

\newcommand{\nablatilde}{\widetilde{\nabla}}

\newcommand{\wbar}{\overline{w}}

\newcommand{\htilde}{\widetilde{h}}

\newcommand{\Utilde}{\widetilde{U}}

\def\DR{\mathop{\rm DR}\nolimits}

\def\cov{\mathop{\rm cov}\nolimits}

\def\an{\mathop{\rm an}\nolimits}

\def\KS{\mathop{\rm KS}\nolimits}

\newcommand{\ubar}{\overline{u}}

\newcommand{\phitilde}{\widetilde{\phi}}

\newcommand{\gammatilde}{\widetilde{\gamma}}
\newcommand{\Gammatilde}{\widetilde{\Gamma}}

\newcommand{\gammabar}{\overline{\gamma}}

\newcommand{\Btilde}{\widetilde{B}}

\newcommand{\tte}{{\tt e}}

\begin{document}
\allowdisplaybreaks

\newcommand{\arXivNumber}{1903.03264}

\renewcommand{\thefootnote}{}

\renewcommand{\PaperNumber}{048}

\FirstPageHeading

\ShortArticleName{Triply Periodic Monopoles and Difference Modules on Elliptic Curves}

\ArticleName{Triply Periodic Monopoles and Difference Modules\\ on Elliptic Curves\footnote{This paper is a~contribution to the Special Issue on Integrability, Geometry, Moduli in honor of Motohico Mulase for his 65th birthday. The full collection is available at \href{https://www.emis.de/journals/SIGMA/Mulase.html}{https://www.emis.de/journals/SIGMA/Mulase.html}}}

\Author{Takuro MOCHIZUKI}

\AuthorNameForHeading{T.~Mochizuki}

\Address{Research Institute for Mathematical Sciences, Kyoto University, Kyoto, 606-8502, Japan}
\Email{\href{mailto:takuro@kurims.kyoto-u.ac.jp}{takuro@kurims.kyoto-u.ac.jp}}

\ArticleDates{Received October 29, 2019, in final form May 18, 2020; Published online June 03, 2020}

\Abstract{We explain the correspondences between twisted monopoles with Dirac type singularity and polystable twisted mini-holomorphic bundles with Dirac type singularity on a 3-dimensional torus. We also explain that they are equivalent to polystable parabolic twisted difference modules on elliptic curves.}


\Classification{53C07; 58E15; 14D21; 81T13}

\renewcommand{\thefootnote}{\arabic{footnote}}
\setcounter{footnote}{0}

\section{Introduction}

We studied the Kobayashi--Hitchin correspondences for singular monopoles with periodicity
in one direction~\cite{Mochizuki-difference-modules}
or two directions~\cite{Mochizuki-q-difference-modules}.
In this paper, we study singular monopoles with periodicity in three directions.
In the analytic aspect, this case is much simpler than the other cases
because a $3$-dimensional torus is compact.
But, there still exist interesting correspondences
with algebro-geometric objects.
Moreover, everything is generalized
to the twisted case. (See Section~\ref{section;19.10.17.10} for the twisted objects.)
Though we also study a generalization to the twisted case,
this introduction is devoted to explain the results in the untwisted case.

\subsection{Triply periodic monopoles with Dirac type singularity}

Let $Y$ be an oriented $3$-dimensional $\real$-vector space
with an Euclidean metric $g_Y$.
Let $\Gamma$ be a~lattice of $Y$.
We set $\nbigm:=Y/\Gamma$,
which is equipped with the induced metric
$g_{\nbigm}$.
Let $Z$ be a~finite subset of $\nbigm$.
Let $E$ be a~$C^{\infty}$-vector bundle on $\nbigm\setminus Z$
with a Hermitian metric $h$,
a~unitary connection $\nabla$
and an anti-self-adjoint endomorphism $\phi$.
A tuple $(E,h,\nabla,\phi)$ is called a monopole
on $\nbigm\setminus Z$
if the Bogomolny equation
\[
 F(\nabla)=\ast\nabla\phi
\]
is satisfied, where $F(\nabla)$ denotes the curvature of $\nabla$,
and $\ast$ denotes the Hodge star operator
with respect to $g_{\nbigm}$.
A point of $P\in Z$ is called a Dirac type singularity of
the monopole $(E,h,\nabla,\phi)$
if $|\phi_Q|_h=O\big(d(Q,Z)^{-1}\big)$ for any $Q\in\nbigm\setminus Z$,
where $\phi_Q$ denotes the element of
the fiber $\End(E)_{|Q}$ over $Q$ induced by $\phi$,
and $d(Q,Z)$ denotes the distance between $Q$ and $Z$.
 Note that the notion of Dirac type singularity was
originally introduced by
Kronheimer~\cite{Kronheimer-Master-Thesis}.
The above condition is equivalent to the original definition,
according to~\cite{Mochizuki-Yoshino}.

\subsection{Mini-holomorphic bundles with Dirac type singularity}\label{subsection;19.3.2.20}

Let us explain a correspondence between monopoles with Dirac type singularity and polystable mini-holomorphic bundles with Dirac type singularity on a $3$-dimensional torus. (See Section~\ref{section;19.10.17.10} below for more details on the notions of mini-complex structures
and mini-holomorphic bundles with Dirac type singularity on $3$-dimensional manifolds.) It was formulated by Kontsevich and Soibelman~\cite{Kontsevich-Soibelman}.

\subsubsection{Mini-complex structure}

We take a linear coordinate system $(x_1,x_2,x_3)$ on $Y$
compatible with the orientation such that $g_Y=\sum dx_i\,dx_i$,
and we set $t:=x_1$ and $w=x_2+\sqrt{-1}x_3$.
The coordinate system induces a~mini-complex structure
on $\nbigm\setminus Z$.
A $C^{\infty}$-function~$f$ on an open subset of $\nbigm$
is called mini-holomorphic
if $\del_tf=\del_{\wbar}f=0$.
Let $\nbigo_{\nbigm\setminus Z}$ denote the sheaf of
mini-holomorphic functions on~$\nbigm\setminus Z$.

\subsubsection{Mini-holomorphic bundles with Dirac type singularity}

Let $\nbigv$ be a locally free $\nbigo_{\nbigm\setminus Z}$-module.
Let $P$ be a point of~$Z$.
We take a lift $(t_0,w_0)\in Y$ of~$P$.
Let $\epsilon$ and $\delta$ denote small positive numbers.
Set
$B_{w_0}(\delta):=
 \bigl\{w\in\cnum\,\big|\,
 |w-w_0|<\delta
 \bigr\}$ and
$B_{w_0}^{\ast}(\delta):=
 \bigl\{w\in\cnum\,\big|\,
 0<|w-w_0|<\delta
 \bigr\}$.
For any $t_0-\epsilon<t<t_0+\epsilon$,
the restriction
$\nbigv_{|\{t\}\times B^{\ast}_{w_0}(\delta)}$
is naturally a locally free $\nbigo_{B^{\ast}_{w_0}(\delta)}$-module.
If $t\neq t_0$,
they extend to locally free $\nbigo_{B_{w_0}(\delta)}$-modules
$\nbigv_{|\{t\}\times B_{w_0}(\delta)}$.
Because mini-holomorphic functions are constant
in the $t$-direction,
we obtain an isomorphism of
$\nbigo_{B^{\ast}_{w_0}(\delta)}$-modules
$\nbigv_{|\{t_0-\epsilon_1\}\times B^{\ast}_{w_0}(\delta)}
\simeq
 \nbigv_{|\{t_0+\epsilon_1\}\times B^{\ast}_{w_0}(\delta)}$
for $0<\epsilon_1<\epsilon$.
If it is meromorphic at $w_0$,
then $P$ is called a Dirac type singularity of $\nbigv$.
If every point of $Z$ is Dirac type singularity,
then $\nbigv$ is called a mini-holomorphic bundle
with Dirac type singularity on~$(\nbigm;Z)$.

\subsubsection{Stability condition}

Kontsevich and Soibelman \cite{Kontsevich-Soibelman}
introduced a sophisticated way
to define a stability condition for
mini-holomorphic bundles
with Dirac type singularity on $(\nbigm;Z)$.

Let $H^j(\nbigm\setminus Z)$
denote the $j$-th cohomology group of $\nbigm\setminus Z$
with $\real$-coefficient.
Let $H_j(\nbigm,Z)$ denote the relative $j$-th homology group
of $(\nbigm,Z)$ with $\real$-coefficient.
Note that there exists the natural isomorphism
\[
 \Phi_Z\colon \ H^2(\nbigm\setminus Z) \simeq H_1(\nbigm,Z).
\]

Let $\gbigt$ denote the space of left invariant vector fields on $\nbigm$,
and let $\gbigt^{\lor}$ denote the left invariant $1$-forms on $\nbigm$.
Let $\sigma$ denote the image of $1$ via the canonical morphism
$\real\lrarr\gbigt\otimes\gbigt^{\lor}$.
It is described as
$\sigma=\sum_{i=1,2,3} \del_{x_i}\otimes dx_i$.

For any mini-holomorphic bundle with Dirac type singularity
$\nbigv$ on $(\nbigm;Z)$,
we obtain
$c_1(\nbigv)\in H^2(\nbigm\setminus Z)$,
and hence
$\Phi_Z(c_1(\nbigv))\in H_1(\nbigm,Z)$.
Then, we obtain the following invariant vector field
\[
\int_{\Phi_Z(c_1(\nbigv))}
 \sigma=
 \sum_{i=1,2,3}
 \left(
 \int_{\Phi_Z(c_1(\nbigv))}dx_i
 \right)\del_{x_i}
\in\gbigt.
\]
Kontsevich and Soibelman discovered that
$\int_{\Phi_Z(c_1(\nbigv))}\sigma$
is a scalar multiplication of $\del_t=\del_{x_1}$,
and they define the degree
$\deg^{\KS}(\nbigv)$ for $\nbigv$
as follows
\[
 \int_{\Phi_Z(c_1(\nbigv))}\sigma=\deg^{\KS}(\nbigv)\del_t.
\]
They introduced the following stability condition.

\begin{Definition}
A mini-holomorphic bundle with Dirac type singularity
$\nbigv$ on $(\nbigm;Z)$ is called stable (resp. semistable)
if
\begin{gather*}
 \deg^{\KS}(\nbigv')/\rank(\nbigv')
<\deg^{\KS}(\nbigv)/\rank(\nbigv)
\\
\bigl(
 \mbox{\rm resp. }
 \deg^{\KS}(\nbigv')/\rank(\nbigv')
\leq\deg^{\KS}(\nbigv)/\rank(\nbigv)
\bigr)
\end{gather*}
for any locally free $\nbigo_{\nbigm\setminus Z}$-submodule
$\nbigv'$ of $\nbigv$ such that
$0<\rank(\nbigv')<\rank(\nbigv)$.
It is called polystable
if it is semistable and a direct sum of stable submodules.
\end{Definition}

\subsubsection{Kobayashi--Hitchin correspondence}

Let $(E,h,\nabla,\phi)$ be a monopole with Dirac type singularity
on $\nbigm\setminus Z$.
We set $\del_{E,\wbar}:=\nabla_{\wbar}$
and $\del_{E,\del_t}:=\nabla_t-\sqrt{-1}\phi$.
Let $\nbigv$ be the sheaf of sections $s$ of $E$
such that $\del_{E,\wbar}s=\del_{E,t}s=0$.
It is a~standard fact that~$\nbigv$ is
a mini-holomorphic bundle with Dirac type singularity
on $(\nbigm;Z)$.
The following theorem was formulated by
Kontsevich and Soibelman~\cite{Kontsevich-Soibelman}.
\begin{Theorem}[the untwisted case in Theorem~\ref{thm;19.2.27.10}, Proposition~\ref{prop;19.2.27.21}]
\label{thm;19.3.2.1}
The procedure induces
an equivalence between
monopoles with Dirac type singularity
on $\nbigm\setminus Z$
and polystable mini-holomorphic bundles with Dirac type singularity
of degree $0$ on $(\nbigm;Z)$.
\end{Theorem}

We shall relate the degree of Kontsevich and Soibelman
with the analytic degree defined in terms of Hermitian metrics
(Proposition~\ref{prop;19.2.27.21}).
Then,
Theorem~\ref{thm;19.3.2.1} follows
from the fundamental theorem due to Simpson~\cite{Simpson88}
as we shall explain in
the proof of Theorem~\ref{thm;19.2.27.10},
which is an analogue of a result due to
Charbonneau and Hurtubise~\cite{Charbonneau-Hurtubise}
for singular monopoles on $3$-dimensional manifolds
obtained as the product of~$S^1$ and a compact Riemann surface.
See also the work of Yoshino~\cite{Yoshino-KH} on the Kobayashi--Hitchin correspondence for monopoles with Dirac type singularity on mini-complex 3-dimensional manifolds.

\subsection{Parabolic difference modules on elliptic curves}

Let us give a complement
on correspondences between mini-holomorphic bundles
with Dirac type singularity
on a $3$-dimensional torus
and parabolic difference modules on elliptic curves.

\begin{Remark}
After completing the first version of this paper,
the author was informed that
{\rm\cite{Kontsevich-Soibelman}} also already contains
the correspondence with difference modules
on elliptic curves.
\end{Remark}

\subsubsection{Parabolic difference modules on elliptic curves
and a stability condition}
\label{subsection;19.10.18.20}

Let $\Gamma_0$ be a lattice of $\cnum$.
We set $T:=\cnum/\Gamma_0$.
Let $\gminia\in \cnum$.
Let $\Phi\colon T\lrarr T$ be the morphism
induced by $\Phi(z)=z+\gminia$.
Let $D\subset T$ be a finite subset.
Let $\nbigo_T(\ast D)$ denote the sheaf of
meromorphic functions on $T$ which may have
poles along $D$.
For any $\nbigo_T$-module $\nbigf$,
we set $\nbigf(\ast D):=\nbigf\otimes_{\nbigo_T}\nbigo_T(\ast D)$.
A parabolic $\gminia$-difference module on~$T$
consists of the following data
$V_{\ast}=\bigl( V,(\vectau_P,\vecnbigl_P)_{P\in D}\bigr)$:{\samepage
\begin{itemize}\itemsep=0pt
\item
 A locally free $\nbigo_T$-module $V$.
\item
 An isomorphism of $\nbigo_T(\ast D)$-modules
 $V(\ast D)\simeq (\Phi^{\ast})^{-1}(V)(\ast D)$.
\item
A sequence
$0\leq \tau_{P,1}<\tau_{P,2}<\cdots<\tau_{P,m(P)}<1$
 for each $P\in D$.
\item
Lattices $\nbigl_{P,i}$ $(i=1,\ldots,m(P)-1)$
of the stalk $V(\ast D)_{P}$ at each $P\in D$.
We formally set
$\nbigl_{P,0}:=V_{P}$
and $\nbigl_{P,m(P)}:=(\Phi^{\ast})^{-1}(V)_P$
at each $P\in D$.
\end{itemize}
When we fix $(\vectau_P)_{P\in D}$,
it is called a parabolic $\gminia$-difference module
on $(T,(\vectau_P)_{P\in D})$.}

The degree of a parabolic $\gminia$-difference module
$(V,(\vectau_P,\vecnbigl_P)_{P\in D})$
is defined as follows
\begin{gather}\label{eq;19.8.10.1}
 \deg\bigl(V,(\vectau_{P},\vecnbigl_P)_{P\in D}\bigr):=
 \deg(V)
+\sum_{P\in D}
 \sum_{i=1}^{m(P)}
 (1-\tau_{P,i})\deg(\nbigl_{P,i},\nbigl_{P,i-1}).
\end{gather}
Here, we set
$\deg(\nbigl_{P,i},\nbigl_{P,i-1}):=
 \length\bigl(\nbigl_{P,i}/\nbigl_{P,i-1}\cap\nbigl_{P,i}\bigr)
-\length\bigl(\nbigl_{P,i-1}/\nbigl_{P,i-1}\cap\nbigl_{P,i}\bigr)$.
The degree can be rewritten as
\[
 \deg\bigl(V,(\vectau_{P},\vecnbigl_P)_{P\in D}\bigr):=
 \deg(V)
-\sum_{P\in D}
 \sum_{i=1}^{m(P)}
 \tau_{P,i}\deg(\nbigl_{P,i},\nbigl_{P,i-1}),
\]
because $\sum_{P\in D}
 \sum_{i=1}^{m(P)}\deg(\nbigl_{P,i},\nbigl_{P,i-1})=0$.
The slope is defined in the standard way
\[
 \mu(V,(\vectau_P,\vecnbigl_P)_{P\in D}):=
 \deg(V,(\vectau_P,\vecnbigl_P)_{P\in D})/\rank V.
\]

For any $\nbigo_T(\ast D)$-submodule $0\neq V'\subset V$
such that
$V'(\ast D)\simeq (\Phi^{\ast})^{-1}(V')(\ast D)$,
we obtain lattices
$\nbigl'_{P,i}$ of $V'(\ast D)_P$
by setting
$\nbigl'_{P,i}:=\nbigl_{P,i}\cap V'(\ast D)_P$
in $V(\ast D)_P$,
and we obtain a parabolic
$\gminia$-difference module
$(V',(\vectau_P,\vecnbigl'_P)_{P\in D})$.
Such $(V',(\vectau_P,\vecnbigl'_P)_{P\in D})$
is called a parabolic $\gminia$-difference submodule
of $(V,(\vectau_P,\vecnbigl_P)_{P\in D})$.

\begin{Definition}
\label{df;19.8.10.2}
$(V,(\vectau_P,\vecnbigl_P)_{P\in D})$
is called stable (resp. semistable)
if
\begin{gather*}
\mu(V',(\vectau_P,\vecnbigl'_P)_{P\in D})
<\mu(V,(\vectau_P,\vecnbigl_P)_{P\in D})\\
 \bigl(
\mbox{\rm resp. }
\mu(V',(\vectau_P,\vecnbigl'_P)_{P\in D})
\leq\mu(V,(\vectau_P,\vecnbigl_P)_{P\in D})
\bigr)
\end{gather*}
for any parabolic $\gminia$-difference submodules
such that $0<\rank V'<\rank V$.
It is called polystable if it is semistable
and a direct sum of stable objects.
\end{Definition}

\subsubsection{Equivalence}

We return to the situation in Section~\ref{subsection;19.3.2.20}.
We take a generator
$e_i=(a_i,\alpha_i)$ $(i=1,2,3)$
of $\Gamma\subset\real_t\times\cnum_w=Y$,
which is compatible with the orientation of $Y$.
We also assume that~$\alpha_1$ and~$\alpha_2$ generate
a lattice in~$\cnum$
and compatible with the orientation of~$\cnum$.
Let $\Gamma_0$ denote the lattice,
and we set $T:=\cnum/\Gamma_0$.
We set
\[
\gamma:=-\frac{a_1\alphabar_2-a_2\alphabar_1}
 {\alpha_1\alphabar_2-\alpha_2\alphabar_1},
\qquad
 \gminit:=a_3+2\Re(\gamma\alpha_3),
\qquad
 \gminia:=\alpha_3.
\]
It is easy to see that $\gminit>0$.
We define the isomorphism
$F\colon \real_t\times\cnum_w\simeq\real_s\times\cnum_u$
by
\[
 s=t+2\Re(\gamma w), \qquad
 u=w.
\]
Note that the induced action of $\Gamma$
on $\real_s\times\cnum_u$ is expressed as follows:
\[
 e_i(s,u)=(s,u+\alpha_i)\quad (i=1,2),
\qquad
 e_3(s,u)=(s+\gminit,u+\gminia).
\]

We set $\closedopen{0}{\gminit}{}:=\{0\leq s<\gminit\}$.
Let $Z_Y$ be the pull back of $Z$
by $Y\lrarr \nbigm$.
Let $D$ denote the image of the composite
of the following maps:
\[
F(Z_Y)\cap
 \bigl(
 \closedopen{0}{\gminit}\times\cnum_u
 \bigr)
\subset
 \real_s\times\cnum_u
\lrarr \cnum_u\lrarr T.
\]

For any $P\in D$,
we take $u_0\in\cnum$ which is mapped to $P$.
We obtain a sequence
$0\leq s_{P,1}<s_{P,2}<\cdots <s_{P,m(P)}<\gminit$
by the condition:
\[
 \{(s_{P,i},u_0)\,|\,i=1,\ldots,m(P)\}
=F(Z_Y)\cap
 \bigl(
 \closedopen{0}{\gminit}\times\{u_0\}
 \bigr).
\]
It is independent of the choice of $u_0$.
We set
$\tau_{P,i}:=s_{P,i}/\gminit$.

\begin{Proposition}[the untwisted case in Propositions~\ref{prop;19.3.2.10} and~\ref{prop;19.2.27.50}]\label{prop;19.3.3.20}
There exists an equivalence between
parabolic difference modules on
$(T,(\vectau_P)_{P\in D})$
and mini-holomorphic bundles
with Dirac type singularity on $(\nbigm;Z)$.
The equivalence preserves the degree
up to the multiplication of a~positive constant.
As a result, the equivalence preserves
the $($poly$)$stability condition.
\end{Proposition}
See Section~\ref{subsection;19.2.25.12}
for the explicit correspondence.
As a consequence
of Theorem~\ref{thm;19.3.2.1}
and Proposition~\ref{prop;19.3.3.20},
we obtain the following theorem.

\begin{Theorem}\label{thm;19.8.10.110}
We have the equivalence of the following objects:
\begin{itemize}\itemsep=0pt
\item
 Monopoles with Dirac type singularity
 on $\nbigm\setminus Z$.
\item
 Polystable mini-holomorphic bundles with Dirac type singularity
 of degree $0$ on $(\nbigm;Z)$.
\item
 Polystable parabolic difference modules
 of degree $0$ on $(T,(\vectau_{P})_{P\in D})$.
\end{itemize}
Here, $Z$ and $(\vectau_P)_{P\in D}$
are related as above.
\end{Theorem}

This study is partially motivated by
the holomorphic Floer theory~\cite{Kontsevich-Soibelman}
of Kontsevich and Soibelman.
Among other things,
they revisit the Riemann--Hilbert correspondence for $D$-modules
from the viewpoint of symplectic topology,
and they extend it to the context of difference modules
of various types.
Moreover, they propose an analogue of the non-abelian Hodge theory
in the context of difference modules,
where the role of harmonic bundles should be played by monopoles
as in Theorem~\ref{thm;19.8.10.110}.

Though the untwisted case is explained
in this introduction,
we shall study the twisted case,
i.e.,
equivalences of twisted mini-holomorphic bundles,
twisted difference modules,
and twisted monopoles.
We should note that
Kontsevich and Soibelman suggested
that there should exist a~twisted version of
of Theorem~\ref{thm;19.8.10.110}.

\section{Preliminary}\label{section;19.10.17.10}

We introduce the notions of
twisted mini-holomorphic bundles
and twisted monopoles
as ge\-neralizations of
the notions of mini-holomorphic bundles~\cite{Mochizuki-difference-modules}
and monopoles.
We are interested only in the case
where the base manifolds are $3$-dimensional torus.
We also introduce twisted difference modules
on elliptic curves.

\subsection{Mini-complex structure on 3-dimensional manifolds}

Let $(t,w)$ denote the standard coordinate system
on $\real\times\cnum$.
Let $M$ be an oriented $3$-dimensional $C^{\infty}$-manifold.
A mini-complex coordinate system on $M$
is a family of open subsets
$U_{\lambda}$ $(\lambda\in\Lambda)$
equipped with an oriented embedding
$\varphi_{\lambda}\colon U_{\lambda}\lrarr
 \real\times\cnum$
satisfying the following conditions.
\begin{itemize}\itemsep=0pt
\item
 $M=\bigcup_{\lambda\in\Lambda} U_{\lambda}$.
\item
 Let $F_{\lambda,\mu}\colon
 \varphi_{\mu}(U_{\lambda}\cap U_{\mu})
\lrarr
 \varphi_{\lambda}(U_{\lambda}\cap U_{\mu})$
denote the induced diffeomorphism
of open subsets in $\real\times\cnum$.
Note that
$F_{\lambda,\mu}$
is expressed as
$((F_{\lambda,\mu})_t(t,w),(F_{\lambda,\mu})_w(t,w))$
in terms of the coordinate systems.
Then,
it holds that
$\del_t(F_{\lambda,\mu})_w=0$
and $\del_{\wbar}(F_{\lambda,\mu})_w=0$.
\end{itemize}
Two mini-complex coordinate systems
$\{(U_{\lambda},\varphi_{\lambda})\}_{\lambda\in\Lambda}$
and
 $\{(V_{\mu},\psi_{\mu})\}_{\mu\in\Gamma}$
are called equivalent
if their union is also a mini-complex coordinate system.
A mini-complex structure on~$M$
is an equivalence class of mini-complex coordinate systems.
We shall not distinguish a mini-complex structure
and a mini-complex coordinate system
contained in the mini-complex structure.

Suppose that $M$ is equipped with a mini-complex structure.
On a mini-complex coordinate neighbourhood
$(U;t,w)$,
let $T_SU$ denote the subbundle of
the tangent bundle $TU$
generated by~$\del_t$.
By patching $T_SU$ for any mini-complex
coordinate neighbourhoods $(U;t,w)$
we obtain the subbundle $T_SM\subset TM$.

Let $T^{\ast}_SM$ denote the dual bundle of $T_SM$.
Let $T_Q^{\ast}M$
denote the kernel of the natural surjection
$T^{\ast}M\lrarr T^{\ast}_SM$.
It is naturally equipped with a complex structure $J$.
Let $\Omega_Q^{1,0}M\subset T_Q^{\ast}M\otimes\cnum$
(resp.~$\Omega_Q^{0,1}M$)
denote the eigen subbundle
with respect to $J$ corresponding to $\sqrt{-1}$
(resp.~$-\sqrt{-1}$).
We set $\Omega^{0,1}M:=(T^{\ast}M\otimes\cnum)/\Omega_Q^{1,0}M$
and $\Omega^{0,i}M:=\bigwedge^i\Omega^{0,1}M$
for $i=0,1,2$.
Similarly,
we set $\Omega^{1,0}M:=(T^{\ast}M\otimes\cnum)/\Omega_Q^{0,1}M$
and $\Omega^{i,0}M:=\bigwedge^i\Omega^{1,0}M$
for $i=0,1,2$.

Let $\delbar_M$ denote the differential operator
$C^{\infty}(M,\cnum)\lrarr
 C^{\infty}\big(M,\Omega^{0,1}M\big)$
induced by the exterior derivative
and the projection
$T^{\ast}M\otimes\cnum
\lrarr
 \Omega^{0,1}M$.
 The induced operator
 \[
C^{\infty}\big(M,\Omega^{0,1}M\big)\lrarr
 C^{\infty}\big(M,\Omega^{0,2}M\big)
 \]
is also denoted by $\delbar_M$.
Similarly, we obtain the operator
\[
 \del_M\colon \ C^{\infty}\big(M,\Omega^{i,0}M\big)
\lrarr
 C^{\infty}\big(M,\Omega^{i+1,0}M\big).
\]

\subsubsection{Riemannian case}\label{subsection;19.10.19.2}

Suppose that $M$ is also equipped with a Riemannian metric~$g_M$.
Let $T_{S,g_M}^{\ast}M$ denote the orthogonal complement of~$T_Q^{\ast}M$.
We shall naturally identify~$T_{S,g_M}^{\ast}M$
and~$T^{\ast}_SM$.

Because $T^{\ast}M$ and $T_Q^{\ast}M$
are oriented,
$T_{S,g_M}^{\ast}M$ is also oriented.
Let $\eta$ be the unique section of~$T_{S,g_M}^{\ast}M$
in the positive direction
such that the norm of $\eta$ is $1$.
By $\eta$,
$T_{S,g_M}^{\ast}M$ is identified with~$\real\times M$.
If there exists a mini-complex coordinate system
$(U;t,w)$ such that $g_{M|U}=dt\,dt+dw\,d\wbar$,
then $\eta_{|M}=dt$.

We obtain a decomposition
\begin{gather}
\label{eq;19.10.18.1}
 T^{\ast}M\otimes\cnum
=\Omega_Q^{1,0}M\oplus \Omega_Q^{0,1}M
\oplus
 T_{S,g_M}^{\ast}M\otimes\cnum.
\end{gather}
We also obtain the isomorphisms
\[
 \Omega_Q^{1,0}M\oplus
 T_{S,g_M}^{\ast}M\otimes\cnum
\simeq
 \Omega^{1,0}M,
\qquad
 \Omega_Q^{0,1}M\oplus
 T_{S,g_M}^{\ast}M\otimes\cnum
\simeq\Omega^{0,1}M.
\]
If the complex structure $J$ on $T^{\ast}_QM$
is an isometry with respect to $g_M$,
the decomposition (\ref{eq;19.10.18.1})
is orthogonal.

\subsection{Twisted mini-holomorphic bundles}

Let $M$ be a mini-complex $3$-dimensional manifold.
Let $E$ be a $C^{\infty}$-vector bundle on $M$.
We shall always assume that the rank of $E$ is finite.
Let $\varrho\in C^{\infty}\big(M,\Omega^{0,2}M\big)$.
\begin{Definition}
A $\varrho$-twisted mini-holomorphic structure of $E$
is a differential operator
$\delbar_E\colon \allowbreak C^{\infty}(M,E)
\lrarr
 C^{\infty}\big(M,\Omega^{0,1}M\otimes E\big)$
such that the following conditions are satisfied.
\begin{itemize}\itemsep=0pt
\item
 $\delbar_E(fs)=f\delbar_E(s)+(\delbar_Mf)\otimes s$
holds for any $f\in C^{\infty}(M,\cnum)$
and $s\in C^{\infty}(M,E)$.
\item
 The induced operator
 $C^{\infty}\big(M,
 \Omega^{0,1}M\otimes
 E\big)\lrarr C^{\infty}\big(M,\Omega^{0,2}M\otimes E\big)$
is also denoted by $\delbar_E$.
Then, $\delbar_E\circ\delbar_E=\varrho\id_E$ holds.
\end{itemize}
Such $(E,\delbar_E)$ is called
a $\varrho$-twisted mini-holomorphic vector bundle.
If $\varrho=0$, we shall omit the adjective ``$0$-twisted''.
\end{Definition}

\begin{Remark}
A $C^{\infty}$-function $f$ on
an open subset $\nbigu\subset M$
is called mini-holomorphic
if \mbox{$\delbar_Mf=0$}.
Let $\nbigo_M$ denote the sheaf of mini-holomorphic functions.
In the case $\varrho=0$,
mini-holomorphic bundles
are naturally identified with
locally free $\nbigo_M$-modules of finite rank.
Let $(E,\delbar_E)$ be a mini-holomorphic bundle on~$M$.
A local section~$s$ of~$E$ is called mini-holomorphic
if $\delbar_E(s)=0$.
Let $\Etilde$ denote the sheaf of mini-holomorphic sections
of $E$.
Then, it is easy to observe that
$\Etilde$ is a~locally free $\nbigo_M$-module
of finite rank.
This correspondence induces an equivalence
between mini-holomorphic bundles
and locally free $\nbigo_M$-modules of finite rank.
\end{Remark}

\subsubsection{Scattering map}
Let $(E,\delbar_E)$ be a $\varrho$-twisted
mini-holomorphic vector bundle on $M$.
Let $\gamma\colon [0,1]\lrarr M$ be a~$C^{\infty}$-path
such that
$T\gamma(T[0,1])\subset T_SM$.
Then,
$\gamma^{-1}(E)$ is equipped with a connection
induced by the $\varrho$-twisted mini-holomorphic structure $\delbar_E$,
and hence we obtain the induced isomorphism
$E_{\gamma(0)}\simeq E_{\gamma(1)}$.
It is called the scattering map in \cite{Charbonneau-Hurtubise}.

Let $(U;t,w)$ be a mini-complex coordinate neighbourhood of $M$.
Let $\del_{E,t}$ (resp. $\del_{E,\wbar}$)
denote the differential operators of $E_{|U}$
induced by $\delbar_E$ and $\del_t$
(resp. $\del_{\wbar}$).
We have the expression
$\varrho=\varrho_0\,dt\,d\wbar$.
Then, the condition $\delbar_E\circ\delbar_E=\varrho\id_E$ on $U$
is equivalent to
$\bigl[
 \del_{E,t},\del_{E,\wbar}
 \bigr]=\varrho_0\id_E$.
Assume that there exists
$\nu=\nu_t\,dt+\nu_{\wbar}d\wbar
 \in C^{\infty}(U,\Omega^{0,1})$
such that $\delbar\nu=\varrho$ on $U$.
Note that such $\nu$ always exists locally.
On $U$,
we set $\delbar_E^{\nu}=\delbar_E-\nu\,\id_E$.
Then, $(E_{|U},\delbar_E^{\nu})$ is clearly a~mini-holomorphic bundle.

Suppose that $U$ is isomorphic to
$\{t_0<t<t_1\}\times B_{\delta}$,
where $B_{\delta}=\{w\in\cnum\,|\,|w|<\delta\}$.
Take $t_0<b_1<b_2<t_1$.
We obtain the scattering map
$F\colon E_{|\{t=b_1\}\times B_{\delta}}
\simeq
 E_{|\{t=b_2\}\times B_{\delta}}$.
Let $\del_{E,\wbar,b_i}$ denote the operators
on $E_{|\{t=b_i\}\times B_{\delta}}$
by $\del_{E,\wbar}$.

\begin{Lemma}
\label{lem;19.10.19.110}
$F^{\ast}(\del_{E,\wbar,b_2})
=\del_{E,\wbar,b_1}+
 \bigl(
 \int_{b_1}^{b_2}\varrho_0\,dt
\bigr)\id$.
\end{Lemma}

\begin{proof} Take $\nu=\nu_{\wbar}\,d\wbar$ such that $\delbar\nu=\varrho$,
i.e.,
$\del_t\nu_{\wbar}=\varrho_0$.
Then,
$F^{\ast}(\del^{\nu}_{E,\wbar,b_2})
=\del^{\nu}_{E,\wbar,b_1}$
because of $[\del^{\nu}_{E,t},\del^{\nu}_{E,\wbar}]=0$.
Then, the claim of the lemma follows.
\end{proof}

\subsubsection{Twisted mini-holomorphic bundles with Dirac type singularity}

Let $Z\subset M$ be a discrete subset.
Let $(E,\delbar_E)$ be a $\varrho$-twisted
mini-holomorphic bundle on $M\setminus Z$.
Let $P$ be a point of $Z$.
Let $(U;t,w)$ be a mini-complex coordinate
neighbourhood around $P\in Z$.
We may assume $(t(P),w(P))=(0,0)$.
By shrinking $U$,
we assume that
$U\simeq
 \{-2\epsilon<t<2\epsilon\}
\times
 B_{\delta}$
by the mini-complex coordinate system
for some $\epsilon>0$ and $\delta>0$.
Set $B^{\ast}_{\delta}:=B_{\delta}\setminus\{0\}$.
We obtain the scattering map
$F\colon E_{|\{-\epsilon\}\times B^{\ast}_{\delta}}
\simeq
E_{|\{\epsilon\}\times B^{\ast}_{\delta}}$.
\begin{Definition}\label{df;19.10.19.1}
$P$ is a Dirac type singularity of $(E,\delbar_E)$
if $F$ and $F^{-1}$ are $O(|w|^{-N})$ for some $N>0$
with respect to $C^{\infty}$-frames of
$E_{|\{\pm\epsilon\}\times B^{\ast}_{\delta}}$.
If each point of $Z$ is Dirac type singularity of $(E,\delbar_E)$,
we say that $(E,\delbar_E)$ is
a $\varrho$-twisted mini-holomorphic bundle
with Dirac type singularity on $(M;Z)$.
\end{Definition}

Take $\nu=\nu_t\,dt+\nu_{\wbar}d\wbar
\in C^{\infty}\big(U,\Omega^{0,1}\big)$
such that
$\delbar\nu=\varrho$.
We set
$\delbar^{\nu}_{E}:=
 \delbar_{E_{|U}}-\nu\id$
so that
$\big(E_{|U},\delbar^{\nu}_{E}\big)$
is mini-holomorphic.
The scattering map
$F^{\nu}\colon
 E_{|\{-\epsilon\}\times B_{\delta}^{\ast}}
 \simeq
 E_{|\{\epsilon\}\times B_{\delta}^{\ast}}$
for $\delbar^{\nu}_E$
is holomorphic with respect to
$\del^{\nu}_{E,\wbar}$.
Note that
$F^{\nu}=\exp\big(\int_{-\epsilon}^{\epsilon}\nu_t\big)F$.
The condition in Definition~\ref{df;19.10.19.1}
is satisfied if and only if
 $F^{\nu}$ extends to a meromorphic isomorphism
$\big(E_{|\{-\epsilon\}\times B_{\delta}},\del^{\nu}_{E,\wbar,-\epsilon}\big)
 (\ast 0)
\simeq
 (E_{|\{\epsilon\}\times B_{\delta}},\del^{\nu}_{E,\wbar,\epsilon})
 (\ast 0)$,
i.e.,
$P$ is Dirac type singularity of
$\big(E_{|U},\delbar^{\nu}_E\big)$
in the sense of \cite[Section~2.2]{Mochizuki-difference-modules}.

We regard $U$ as an open subset of $\real\times\cnum$
by the coordinate system $(t,w)$.
Let $\varphi\colon \cnum^2\lrarr\real\times\cnum$
be given by
$\varphi(z_1,z_2)=\big(|z_1|^2-|z_2|^2,2z_1z_2\big)$.
Let $\Utilde$ be the pull back of $U$
by $\varphi$.
The mini-holomorphic bundle
$\big(E,\delbar^{\nu}_E\big)_{|U\setminus\{P\}}$
induces an $S^1$-equivariant holomorphic vector bundle
$\big(\Etilde'_P,\delbar^{\nu}_{\Etilde'_P}\big)$ on
$\Utilde\setminus\{(0,0)\}$,
which uniquely extends to an $S^1$-equivariant
holomorphic vector bundle
$\big(\Etilde^{\nu}_P,\delbar^{\nu}_{\Etilde_P}\big)$
on $\Utilde$.
(See \cite[Section~2.2]{Mochizuki-Yoshino} for
a more detailed explanation.)

\begin{Lemma}
Suppose that $\nu_i=\nu_{i,t}\,dt+\nu_{i,\wbar}d\wbar
 \in C^{\infty}\big(U,\Omega^{0,1}\big)$ $(i=1,2)$
satisfy
$\delbar\nu_i=\varrho$.
Then,
the natural identification
$\Etilde^{\nu_1}_{P|\Utilde\setminus\{(0,0)\}}
=\Etilde'_P
=\Etilde^{\nu_2}_{P|\Utilde\setminus\{(0,0)\}}$
uniquely extends to
a $C^{\infty}$-isomorphism
$\Etilde^{\nu_1}_P\simeq
 \Etilde^{\nu_2}_P$.
\end{Lemma}

\begin{proof}
Set $\nu_0=\nu_{0,t}\,dt+\nu_{0,\wbar}\,d\wbar:=\nu_2-\nu_1$.
We have
$\delbar_E^{\nu_2}=
 \delbar_E^{\nu_1}-\nu_0\id_E$.
By the construction
(see \cite[Section~2.2]{Mochizuki-Yoshino}),
we have
$\delbar^{\nu_2}_{\Etilde'_P}
=\delbar^{\nu_1}_{\Etilde'_P}
-\bigl(
 \varphi^{\ast}(\nu_{0,t})\delbar\varphi^{\ast}(t)
+\varphi^{\ast}(\nu_{0,\wbar})\delbar\varphi^{\ast}(\wbar)
 \bigr)\id$.
Then, the claim of the lemma is clear.
\end{proof}

We set
$\Etilde_{P}:=\Etilde^{\nu}_P$
for $\nu\in C^{\infty}\big(U,\Omega^{0,1}\big)$
such that $\delbar\nu=\varrho$,
which is called the Kronheimer resolution of
$\big(E,\delbar_E\big)$ at~$P$.

\begin{Definition}\label{df;19.10.19.100}
A Hermitian metric $h$ of $E$
is called adapted at~$P$
if the induced metric~$\htilde_P$
of~$\Etilde'_P$ extends to a $C^{\infty}$-metric
of the Kronheimer resolution $\Etilde_{P}$.
If~$h$ is adapted at any point of~$Z$,
then $h$ is called an adapted metric of $\big(E,\delbar_E\big)$.
\end{Definition}

\subsubsection{Chern connections and Higgs fields}

Suppose that we are given a splitting
$TM/T_SM\lrarr TM$.
It induces the following decompositions:
\begin{gather}\label{eq;19.10.17.1}
 T^{\ast}M\otimes\cnum
\simeq
 \Omega^{1,0}_QM
\oplus
 \Omega^{0,1}_QM
\oplus
 T^{\ast}_SM\otimes\cnum,
\\
\label{eq;19.10.17.2}
 \Omega^{0,1}M
\simeq
 \Omega^{0,1}_QM\oplus T^{\ast}_SM\otimes\cnum,
\\
\label{eq;19.10.17.3}
 \Omega^{1,0}M
\simeq
 \Omega^{0,1}_QM\oplus T^{\ast}_SM\otimes\cnum.
\end{gather}

Let $(E,\delbar_E)$ be a $\varrho$-twisted mini-holomorphic bundle
on $M$. By~(\ref{eq;19.10.17.2}),
we obtain a decomposition
$\delbar_E=\delbar^S_E\oplus \delbar^Q_E$,
where
$\delbar^S_E(s)\in
 C^{\infty}\big(X,(T_SM\otimes\cnum)^{\lor}\big)$
and
$\delbar^Q_E(s)\in
 C^{\infty}\big(X,\Omega^{0,1}_QM\big)$.

Let $h$ be a Hermitian metric of $E$.
We obtain the differential operator
$\del_{E,h}\colon C^{\infty}(X,E)\lrarr C^{\infty}\big(X,\Omega^{1,0}M\otimes E\big)$
satisfying the condition
$\delbar_Mh(u,v)=h(\delbar_Eu,v)+h(u,\del_{E,h}v)$
for any $u,v\in C^{\infty}(X,E)$.
We also obtain the decomposition
$\del_{E,h}=\del_{E,h}^Q+\del_{E,h}^S$
induced by~(\ref{eq;19.10.17.3}).
For a mini-complex coordinate neighbourhood
$(U;t,w)$,
we obtain the operators
$\del_{E,h,w}$
(resp. $\del_{E,h,t}$) on $E$
induced by
$\del_{E,h}$ and $\del_w$
(resp.~$\del_t$).
\begin{Remark}
In \cite{Mochizuki-difference-modules}, $\del_{E,h,t}$ is denoted
as $\del_{E,h,t}'$.
\end{Remark}

By using~(\ref{eq;19.10.17.1}), we set
\[
 \nabla_h:=
 \delbar_{E}^Q+\del_{E,h}^Q
+\frac{1}{2}
 \bigl(
 \delbar^S_E
+\del^S_{E,h}
 \bigr),
\qquad
 \phi_h:=
 \frac{\sqrt{-1}}{2}
 \bigl(
 \delbar_E^S-\del_{E,h}^S
 \bigr).
\]
They are called the Chern connection and the Higgs field
of $(E,\delbar_E,h)$.
Note that they depend on the choice of a splitting
$TM/T_SM\lrarr TM$.

If $M$ is also equipped with a Riemannian metric $g_M$,
we shall use the splitting $TM/T_SM\lrarr TM$ induced by $g_M$.
Moreover, by the section $\eta$ in Section~\ref{subsection;19.10.19.2},
$T^{\ast}_{S,g_M}M$ is identified with
the product bundle $\real\times M$.
Hence, we regard $\phi_h$
as an anti-Hermitian endomorphism of $E$.
In particular,
if $g_M=dt\,dt+dw\,d\wbar$
on a mini-complex coordinate neighbourhood $(U;t,w)$,
the following holds for any $s\in C^{\infty}(U,E)$:
\begin{gather*}
 \nabla_h(s)
=(\del_{E,\wbar}s)\,d\wbar+(\del_{E,h,w}s)\,dw
+\frac{1}{2}\bigl(
 \del_{E,t}s+\del_{E,h,t}s
 \bigr)\,dt,
\\
 \phi_h(s)=
 \frac{\sqrt{-1}}{2}
 \bigl(
 \del_{E,t}s-\del_{E,h,t}s
 \bigr).
\end{gather*}

\subsection{Twisted monopoles in the locally Euclidean case}

\subsubsection{Twisted monopoles}

Let $(M,g_M)$ be an oriented Riemannian $3$-dimensional manifold.
Let $B$ be a real $2$-form on $M$.
Let $E$ be a vector bundle on $M$
equipped with a Hermitian metric $h$,
a unitary connection $\nabla$,
and an anti-Hermitian endomorphism $\phi$.
\begin{Definition}
Such a tuple $(E,h,\nabla,\phi)$ is called
a $B$-twisted monopole
if the following $B$-twisted Bogomolny equation is
satisfied:
\[
F(\nabla)=\ast\nabla\phi+\sqrt{-1}B\id_E.
\]
Here $F(\nabla)$ denotes the curvature of $\nabla$, and $\ast$ denotes the Hodge star operator.
\end{Definition}

Let $A$ and $f$ be a real $1$-form and an $\real$-valued $C^{\infty}$-function on $M$, respectively.
We set
$\nablatilde:=\nabla+\sqrt{-1}A\,\id$,
$\phitilde:=\phi+\sqrt{-1}f\,\id$
and
$\Btilde:=B+dA-\ast(df)$.
Then, the following is easy to see.
\begin{Lemma}
\label{lem;19.10.19.30}
$(E,h,\nabla,\phi)$ is a $B$-twisted monopole
if and only if
$\big(E,h,\nablatilde,\phitilde\big)$
is a $\Btilde$-twisted monopole.
\end{Lemma}

\begin{Remark}\label{rem;19.10.19.31}
If $M$ is compact,
any real $2$-form $B$ on $M$ is expressed as
$B=dA-\ast df+B_0$,
where $A$ is a real $1$-form,
$f$ is a $\real$-valued $C^{\infty}$-function,
and $B_0$ is a harmonic $1$-form.
Indeed,
let~$G$ denote the Green operator
for the Laplace-Beltrami operator
on the space of $2$-forms on~$M$.
Then, $B-(d^{\ast}d+dd^{\ast})G(B)$ is a harmonic $2$-form,
and $G(B)$ is $C^{\infty}$.
We can also deduce that
for any point $P\in M$,
there exists a neighbourhood $M_P$ of $P$
such that $B_{|M_P}=dA_P-\ast df_P$
for a real $1$-form $A_P$ and an $\real$-valued
$C^{\infty}$-function $f_P$ on~$M_P$.
\end{Remark}

\subsubsection[Twisted monopoles and twisted mini-holomorphic bundles in the locally Euclidean case]{Twisted monopoles and twisted mini-holomorphic bundles\\ in the locally Euclidean case}
\label{subsection;19.10.19.32}

Suppose that $M$ is also equipped with a mini-complex structure. Moreover, we assume that $M$ is a locally Euclidean,
i.e.,
for each $P\in M$,
there exists a mini-complex coordinate neighbourhood
$(U;t,w)$ of $P$
such that the Riemannian metric of $M$ on $U$
is $dt\,dt+dw\,d\wbar$.
Note that $d\eta=0$
for the global trivialization $\eta$
of $T_{S,g_M}^{\ast}M$
in Section~\ref{subsection;19.10.19.2}.
By~(\ref{eq;19.10.18.1}),
for any complex vector bundle $\nbigv$ on $M$,
we obtain the decomposition
\begin{gather}\label{eq;19.10.19.3}
\nbigv\otimes
\bigwedge^2\bigl(
 T^{\ast}M\otimes\cnum
 \bigr)
=\bigl(
 \nbigv\otimes
 \Omega^{1,0}_QM\wedge\eta
\bigr)
\oplus
 \bigl(
 \nbigv\otimes
 \Omega^{0,1}_QM\wedge\eta
 \bigr)
\oplus
 \bigl(
 \nbigv\otimes
 \Omega^{1,1}_QM
\bigr),
\end{gather}
where
$\Omega^{1,1}_QM:=
 \Omega_Q^{1,0}M\wedge
 \Omega_Q^{0,1}M$.
For any section $s$ of
$\nbigv\otimes
 \bigwedge^2\bigl(
 T^{\ast}M\otimes\cnum
 \bigr)$,
we obtain the decomposition
$s=s^{(1,0),\eta}+s^{(0,1),\eta}+s^{(1,1)}$
according to (\ref{eq;19.10.19.3}).
In particular,
we obtain the decomposition
$B=B^{(1,0),\eta}
+B^{(0,1),\eta}+B^{(1,1)}$.
Because $B$ is real,
$B^{(1,1)}$ is also real,
and $B^{(1,0),\eta}=\overline{B^{(0,1),\eta}}$ holds.

We can check the following lemma
by a direct computation.
\begin{Lemma}
Let $(E,h,\nabla,\phi)$ be a $B$-twisted monopole on $M$.
We have the decomposition
$\nabla=\nabla_Q^{1,0}+\nabla_Q^{0,1}+\nabla_S$
induced by {\rm(\ref{eq;19.10.18.1})}.
We set
$\delbar_{E}:=\nabla_Q^{0,1}+\nabla_S-\sqrt{-1}\phi \eta$.
Then, $(E,\delbar_E)$
is a $\sqrt{-1}B^{(0,1),\eta}$-twisted mini-holomorphic bundle.
\end{Lemma}

Conversely,
let $(E,\delbar_E)$ be a $\varrho$-twisted mini-holomorphic bundle
on $M$.
Let $h$ be a Hermitian metric of $E$.
We obtain the Chern connection $\nabla_h$
and the Higgs field $\phi_h$.

\begin{Lemma}
\label{lem;19.10.20.2}
We have
$\bigl(
 F(\nabla_h)-\ast\nabla_h\phi_h
 \bigr)^{(0,1),\eta}
=\varrho\id_E$
and
$\bigl(
 F(\nabla_h)-\ast\nabla_h\phi_h
 \bigr)^{(1,0),\eta}
=-\overline{\varrho}\id_E$.
\end{Lemma}
\begin{proof} We have
$\delbar_E=\nabla_{h,Q}^{0,1}+\nabla_{h,S}-\sqrt{-1}\phi_h\eta$
and
$\del_{E,h}=\nabla_{h,Q}^{1,0}+\nabla_{h,S}+\sqrt{-1}\phi_h\eta$.
Because
$\delbar_E\circ\delbar_E=\varrho\id$,
we obtain
$\del_{E,h}\circ\del_{E,h}=-\overline{\varrho}\id$.
Then, we obtain the claim of the lemma by computations.
\end{proof}

\begin{Corollary}\label{cor;19.10.20.1}
There exists a real $2$-form $B$
such that
$(E,h,\nabla_h,\phi_h)$ is a $B$-twisted monopole
if and only if
the trace-free part of
$\bigl(F(\nabla_h)-\ast\nabla_h\phi_h\bigr)^{(1,1)}$ is $0$,
i.e., there exists a real $2$-form $\varpi$
such that
$\bigl(F(\nabla_h)-\ast\nabla_h\phi_h\bigr)^{(1,1)}
=\sqrt{-1}\varpi\id$.
In that case,
$B=-\sqrt{-1}(\varrho-\overline{\varrho})+\varpi$.
\end{Corollary}

\begin{Remark}If the condition in Corollary~\ref{cor;19.10.20.1}
is satisfied,
$\big(E,\delbar_E,h\big)$ is also called
a $B$-twisted monopole.
\end{Remark}

\subsubsection{Dirac type singularity}

Let $Z$ be a discrete subset of $M$.
Let $B$ be a real $2$-form on $M$.
Let $(E,h,\nabla,\phi)$ be a $B$-twisted monopole
on $M\setminus Z$.
Let $(E,\delbar_E)$ be the underlying
$\sqrt{-1}B^{(0,1),\eta}$-twisted mini-holomorphic bundle.

\begin{Definition}\label{df;19.10.19.20}
A point $P\in Z$ is called Dirac type singularity
of $(E,h,\nabla,\phi)$
if the following conditions are satisfied:
\begin{itemize}\itemsep=0pt
\item
$P$ is Dirac type singularity of
$\big(E,\delbar_E\big)$.
\item
$h$ is an adapted metric of
$\big(E,\delbar_E\big)$
 in the sense of Definition~\ref{df;19.10.19.100}.
\end{itemize}

We say that $(E,h,\nabla,\phi)$ is a $B$-twisted monopole
with Dirac type singularity on $(M;Z)$
if any point $P\in Z$ is Dirac type singularity of $(E,h,\nabla,\phi)$.
\end{Definition}

\begin{Lemma}$P$ is Dirac type singularity of
$(E,h,\nabla,\phi)$
 if and only if
 there exists a neighbourhood $M_P$ of $P$
 in~$M$ such that
 $|\phi_Q|_h=O\bigl(d(P,Q)^{-1}\bigr)$
 for $Q\in M_P\setminus\{P\}$.
\end{Lemma}
\begin{proof} If $M_P$ is sufficiently small,
there exists a real $1$-form $A_P$
and an $\real$-valued $C^{\infty}$-function~$f_P$
such that $B_{|M_P}=dA_P-\ast df_P$.
Then,
$(\Etilde,\htilde):=
(E,h)_{|M_P\setminus\{P\}}$
with
$\nablatilde:=\nabla-\sqrt{-1}A_P\id_E$
and $\phitilde:=\phi-\sqrt{-1}f_P\id_E$
is a monopole on $M_P\setminus\{P\}$.
If $P$ is Dirac type singularity of
$(E,h,\nabla,\phi)$,
then $P$ is Dirac type singularity of
$(\Etilde,\htilde,\nablatilde,\phitilde)$.
According to~\cite{Mochizuki-Yoshino},
it is equivalent to
$|\phitilde_Q|_h=O\bigl(d(P,Q)^{-1}\bigr)$
around any point $P\in Z$,
which is equivalent to
$|\phi_Q|_h=O\bigl(d(P,Q)^{-1}\bigr)$
around any point $P\in Z$.
\end{proof}

\subsection{Twisted difference modules}

Let $\Gamma_0\subset \cnum$ be a lattice.
We put $T:=\cnum/\Gamma_0$.
Take any $\gminia\in T$,
and define the automorphism $\Phi$ of $T$
by $\Phi(z)=z+\gminia$.
Let $\gbigl$ be a holomorphic line bundle
of degree $0$ on $T$.

A parabolic $\gbigl$-twisted difference module
$V_{\ast}=(V,(\vectau_P,\vecnbigl_P)_{P\in D})$
on $T$
consists of the following data:
\begin{itemize}\itemsep=0pt
\item
A locally free $\nbigo_T$-module $V$
equipped with an isomorphism
$V\otimes\nbigo_T(\ast D)
\simeq (\Phi^{\ast})^{-1}(V)\otimes\gbigl\otimes\nbigo_T(\ast D)$,
where $D$ is a finite subset of $T$.
\item
A sequence
$0\leq \tau_{P,1}<\tau_{P,2}<\cdots<\tau_{P,m(P)}<1$
 for each $P\in D$.
\item
Lattices $\nbigl_{P,i}$ $(i=1,\ldots,m(P)-1)$
of the stalk $V(\ast D)_{P}$ at each $P\in D$.
We formally set
$\nbigl_{P,0}:=V_{P}$
and $\nbigl_{P,m(P)}:=\bigl(
 (\Phi^{\ast})^{-1}(V)\otimes\gbigl\bigr)_P$
 at each $P\in D$.
\end{itemize}
The degree of $V_{\ast}$ is defined
by the formula~(\ref{eq;19.8.10.1}), i.e.,{\samepage
\[
 \deg(V_{\ast}):=
 \deg(V)
+\sum_{P\in D}
 \sum_{i=1}^{m(P)}
 (1-\tau_{P,i})\deg(\nbigl_{P,i},\nbigl_{P,i-1}).
\]
We set
$\mu(V_{\ast}):=\deg(V_{\ast})/\rank(V)$.}

For any $\nbigo_T(\ast D)$-submodule
$0\neq V'\subset V$
such that
$V'(\ast D)\simeq (\Phi^{\ast})^{-1}(V')(\ast D)$,
we obtain lattices
$\nbigl'_{P,i}$ of $V'(\ast D)_P$
by setting
$\nbigl'_{P,i}:=\nbigl_{P,i}\cap V'(\ast D)_P$
in $V(\ast D)_P$,
and we obtain a parabolic
$\gbigl$-twisted
$\gminia$-difference module
$V'_{\ast}=(V',(\vectau_P,\vecnbigl'_P)_{P\in D})$.
Such $V'_{\ast}$
is called a parabolic $\gminia$-difference submodule
of $V_{\ast}$.

\begin{Definition}
$V_{\ast}$
is called stable (resp. semistable)
if
\[
\mu(V'_{\ast})
<\mu(V_{\ast})
\qquad
 \bigl(
\mbox{\rm resp. }
\mu(V'_{\ast})
\leq\mu(V_{\ast})
\bigr)
\]
for any
parabolic $\gminia$-difference submodules $V'$
such that $0<\rank V'<\rank V$.
It is called polystable if it is semistable
and a direct sum of stable objects.
\end{Definition}

\subsubsection{Example}

It is easy to construct examples of parabolic difference modules.

\begin{Lemma}\label{lem;19.10.19.300}
For any holomorphic line bundle $\gbigl$ of degree $0$,
and for any $d\in\real$,
there exists a parabolic $\gbigl$-difference module $V_{\ast}$
of rank one
such that $\deg(V_{\ast})=d$.
\end{Lemma}

\begin{proof}
There exist $P_1,\ldots,P_{n}\in T$
and $\ell_i\in\seisuu$
such that
$\gbigl\big(\sum_{i=1}^n\ell_i P_i\big)=\nbigo_T$.
Note that $\sum\ell_i=0$.
We take $P_0\in T\setminus\{P_1,\ldots,P_{n}\}$.
We set $D:=\{P_0,P_1,\ldots,P_n\}$.
We set $V:=\nbigo_T$.
By our choice of $D$,
there exists an isomorphism
$F\colon V(\ast D)\simeq (\Phi^{\ast})^{-1}(V)\otimes\gbigl(\ast D)$.
We set $m(P_i)=1$
and $\tau_{P_i,1}=0$ for $i=1,\ldots,n$.
We set $m(P_0)=2$,
and we choose $0\leq \tau_{P_0,1}<\tau_{P_0,2}<1$.
We set $\nbigl_{P_0,1}=\nbigo_{T}(\ell P_0)_{P_0}$
for an integer $\ell$.
Then, we obtain a parabolic $\gbigl$-twisted
difference module
$V_{\ast}^{(\ell,\tau_{P_0,1},\tau_{P_0,2})}$
for which
$\deg\big(V_{\ast}^{(\ell,\tau_{P_0,1},\tau_{P_0,2})}\big)
=(\tau_{P_0,2}-\tau_{P_0,1})\ell$.
Then, the claim is clear.
\end{proof}

\section{Equivalences}\label{section;19.10.19.400}

We shall study equivalences of
twisted mini-holomorphic bundles,
twisted difference modules,
and twisted monopoles.
First, in Section~\ref{subsection;20.5.17.1},
we introduce analytically stability condition
for twisted mini-holomorphic bundles
in terms of adapted metrics.
We also prepare some formulas for the curvature
and the Higgs field
of a twisted mini-holomorphic bundle with a Hermitian metric
which are standard in the context of
mini-holomorphic bundles as in \cite{Mochizuki-difference-modules}.
In Section~\ref{subsection;19.8.10.50},
we shall explain the equivalence between
twisted mini-holomorphic bundles
and twisted difference modules,
which preserves the stability conditions.
In Section~\ref{subsection;20.5.17.2},
we shall explain the equivalence between
polystable twisted mini-holomorphic bundles
and twisted monopoles.

\subsection{Analytic stability condition for twisted mini-holomorphic bundles}\label{subsection;20.5.17.1}

\subsubsection{3-dimensional torus with mini-complex structure}\label{subsection;19.2.27.30}

We take an oriented base
$(a_i,\alpha_i)$ $(i=1,2,3)$
of the $\real$-vector space
$\real\times\cnum$.
Let $Y:=\real\times\cnum$
with the Riemannian metric $dt\,dt+dw\,d\wbar$.
It is equipped with the mini-complex structure
induced by the mini-complex coordinate system $(t,w)$.
We consider the action of
$\Gamma:=\seisuu\tte_1\oplus\seisuu\tte_2\oplus\seisuu\tte_3$
on $Y$ given by
\[
 \tte_i(t,w)=(t,w)+(a_i,\alpha_i), \qquad i=1,2,3.
\]
Let $\nbigm$ denote the quotient space
of $Y$ by the action of $\Gamma$.
It is equipped with a naturally induced mini-complex structure.

\subsubsection{Contraction of the curvature}

Let $Z$ be a finite subset of $\nbigm$.
Take $\varrho\in C^{\infty}(\nbigm,\Omega^{0,2}\nbigm)$.
Let $(E,\delbar_E)$ be a $\varrho$-twisted
mini-holomorphic bundle
on $\nbigm\setminus Z$.
Let $h$ be a Hermitian metric of~$E$.
As in~\cite{Mochizuki-difference-modules},
we set
\begin{gather*}
 G(h):=
 \bigl[
 \nabla_{h,w},
 \nabla_{h,\wbar}
 \bigr]
-\frac{\sqrt{-1}}{2}\nabla_{h,t}\phi_h.
\end{gather*}
If we emphasize the dependence on $\delbar_E$,
we use the notation $G(h,\delbar_E)$.
Note that
\begin{gather}
\label{eq;19.10.20.3}
 G(h)\,dw\,d\wbar=
 \bigl(F(\nabla_h)-\ast\nabla_h\phi_h\bigr)^{(1,1)}
\end{gather}
for the notation in Section~\ref{subsection;19.10.19.32}.

Let $U$ be an open subset of $\nbigm\setminus Z$
with $\nu=\nu_t\,dt+\nu_{\wbar}\,d\wbar
\in C^{\infty}\big(U,\Omega^{0,1}\big)$.
On $U$,
we set
$\delbar^{\nu}_E:=\delbar_E-\nu\id_E$.
Then,
$\big(E_{|U},\delbar^{\nu}_E\big)$
is a
$(\varrho_{|U}-\delbar\nu)$-twisted
mini-holomorphic bundle on $U$.
We obtain the Chern connection
$\nabla^{\nu}_h$
and the Higgs field $\phi^{\nu}_h$.

\begin{Lemma}
\label{lem;19.10.19.101}
The following holds:
\begin{gather*}
 \phi_h^{\nu}=\phi_h-\sqrt{-1}\Re(\nu_t)\id_E,
\qquad
 \nabla_h^{\nu}
=\nabla_h
-\sqrt{-1}\bigl(
 2\Image(\nu_{\wbar}d\wbar)
+\Image(\nu_t\,dt)
 \bigr)\id_E,
\\
 \nabla_h^{\nu}\big(\phi_h^{\nu}\big)
=\nabla_h(\phi_h)-\sqrt{-1}d\Re(\nu_t)\id_E,
\\
 F\big(\nabla^{\nu}_h\big)
=F(\nabla_h)-
\sqrt{-1}d\bigl(
 2\Image(\nu_{\wbar}dw)
+\Image(\nu_t)dt
\bigr)\id_E,
\\
 G\big(h,\delbar_E^{\nu}\big)
=G(h,\delbar_E)-
\bigl(
 2\Re(\del_w\nu_{\wbar})
+2^{-1}\Re(\del_t\nu_t)
\bigr)\,\id_E.
\end{gather*}
\end{Lemma}
We can check the formulas by direct computations.

Let $E'$ be any $\varrho$-twisted mini-holomorphic subbundle of $E$,
i.e.,
$\delbar_EC^{\infty}(\nbigm\setminus Z,E')
\subset
C^{\infty}\big(\nbigm\setminus Z,\Omega^{0,1}\nbigm\otimes E'\big)$.
We have the natural $\varrho$-twisted mini-holomorphic
structure $\delbar_{E'}$ on~$E'$.
Let $h_{E'}$ be the induced metric of~$E'$.
Let $p_{E'}$ be the orthogonal projection of~$E$
onto $E'$ with respect to~$h$.

\begin{Lemma}The following Chern--Weil formula holds:
\begin{gather}\label{eq;19.2.27.2}
 \Tr G(h_{E'})=
 \Tr\bigl(
 G(h_E)\cdot p_{E'}
 \bigr)
-\bigl|\del_{E,\wbar}p_{E'}\bigr|^2
-\frac{1}{4}\bigl|\del_{E,t}p_{E'}\bigr|^2.
\end{gather}
\end{Lemma}

\begin{proof} If $\varrho=0$,
it is proved in \cite[Section~2.8.2]{Mochizuki-difference-modules}.
Let us study the general case.
It is enough to prove the equality
locally around any point of $Q\in \nbigm\setminus Z$.
On a neighbourhood~$U$ of~$Q$,
there exists
$\nu\in C^{\infty}\big(U,\Omega^{0,1}\big)$
such that
$\delbar_E^{\nu}$
is a mini-holomorphic structure of~$E_{|U}$.
Note that
$\del^{\nu}_{E,\wbar}p_{E'}
=\del_{E,\wbar}p_{E'}$,
$\del^{\nu}_{E,t}p_{E'}
=\del_{E,t}p_{E'}$.
Moreover,
$\big(E',\delbar^{\nu}_{E'}\big)$ is a mini-holomorphic subbundle
of $\big(E,\delbar^{\nu}_E\big)$,
and
$G\big(h_{E'},\delbar^{\nu}_{E'}\big)=G\big(h_{E'},\delbar_{E'}\big)-
\bigl(
 2\Re(\del_w\nu_{\wbar})+2^{-1}\Re(\del_t\nu_t)
 \bigr)\id_{E'}$.
Then, we obtain the desired formula.
\end{proof}

\subsubsection[Analytic stability condition for mini-holomorphic bundles with a Hermitian metric]{Analytic stability condition for mini-holomorphic bundles\\ with a Hermitian metric}

Let $\big(E,\delbar_E\big)$ be a $\varrho$-twisted mini-holomorphic bundle on $\nbigm\setminus Z$
with a Hermitian metric~$h$.

\begin{Definition}\label{df;19.8.10.20} If $\Tr G(h)$ is expressed as a sum of an $L^1$-function and a non-positive function, then we set
$\deg\big(E,\delbar_E,h\big):=\int_{\nbigm\setminus Z}\Tr G(h)\dvol_{\nbigm}
\in\real\cup\{-\infty\}$.
We also set
$\mu\big(E,\delbar_E,h\big):=\deg\big(E,\delbar_E,h\big)/\rank(E)$.
\end{Definition}

Suppose that $|G(h)|_h$ is $L^1$.
By~(\ref{eq;19.2.27.2}),
$\deg(E',h_{E'})$
is defined in $\real\cup\{-\infty\}$
for any $\varrho$-twisted mini-holomorphic subbundle~$E'$
of~$E$.
\begin{Definition}\label{df;19.8.10.21}
Suppose that $|G(h)|_h$ is $L^1$.
Then,
$\big(E,\delbar_E,h\big)$ is called analytically stable
if
$\mu\big(E',\delbar_{E'},h_{E'}\big)
<\mu\big(E,\delbar_E,h\big)$
for any $\varrho$-twisted
mini-holomorphic subbundle
$E'\subset E$
with $0<\rank(E')<\rank(E)$.
\end{Definition}

\subsubsection[Adapted metrics of twisted mini-holomorphic bundles with Dirac type singularity]{Adapted metrics of twisted mini-holomorphic bundles\\ with Dirac type singularity}

Let $\big(E,\delbar_E\big)$ be a $\varrho$-twisted mini-holomorphic bundle with Dirac type singularity on $(\nbigm;Z)$.

\begin{Lemma}\label{lem;19.10.19.200}
If $h$ is an adapted metric at $P$, then $G(h)_Q=O\big(d(P,Q)^{-1}\big)$ around $P$, where $d(P,Q)$ denotes the distance of $P$ and $Q$. In particular, if $h$ is an adapted metric of $\big(E,\delbar_E\big)$,
then $|G(h)|_h$ is $L^1$.
\end{Lemma}
\begin{proof}
In the case $\varrho=0$,
it is proved in \cite[Lemma~2.35]{Mochizuki-difference-modules}.
The general case follows from
Lemma~\ref{lem;19.10.19.101}.
\end{proof}

\begin{Lemma}\label{lem;19.2.27.1}
Let $\big(E,\delbar_E\big)$ be a $\varrho$-twisted mini-holomorphic bundle
with Dirac type singularity on $(\nbigm;Z)$.
Let $E'\neq 0$ be a $\varrho$-twisted
mini-holomorphic subbundle of~$E$.
Let $h$ and $h'$ be adapted Hermitian metrics of
$E$ and $E'$, respectively.
Let $h_{E'}$ be the metric of $E'$
induced by $h$.
Then,
$\deg(E',h_{E'})=\deg(E',h')$.
\end{Lemma}
\begin{proof}
It is enough to study the case $\rank E'=1$.
We may assume that there exist neighbourhoods
$\nbigu_P$ of $P\in Z$
such that
$h_{E'}=h'$ on $\nbigm\setminus\bigcup_{P\in Z}\nbigu_P$.
Then, we have only to prove that
$\int_{\nbigu_P}G(h')
=\int_{\nbigu_P}G(h_{E'})$
for any $P\in Z$.
By Lemma~\ref{lem;19.10.19.101},
it is enough to study the case
$\varrho=0$.
It is proved in the proof of
\cite[Proposition~9.4]{Mochizuki-difference-modules}
(See the argument to compare
 $\int G(h_{0,E_1})$ and $\int G(h_{2,E_1})$
in the proof of \cite[Proposition 9.4]{Mochizuki-difference-modules}.)
\end{proof}

\begin{Corollary}
If $h_1$ and $h_2$ are adapted metrics of
$\big(E,\delbar_E\big)$,
$\deg\big(E,\delbar_E,h_1\big)
=\deg\big(E,\delbar_E,h_2\big)$ holds.
\end{Corollary}

\begin{Lemma}\label{lem;19.10.18.30}
Take a small neighbourhood $\nbigu_P$ of $P\in Z$.
The following estimates hold
for $Q\in \nbigu_P\setminus\{P\}$:
\begin{gather*}
|\phi_{h,Q}|_h=O\bigl(d(P,Q)^{-1}\bigr),
\qquad
 |(\nabla\phi_h)_{Q}|_{h,g_{\nbigm}}=O\bigl(d(P,Q)^{-2}\bigr),
\\
 |F(\nabla_h)_Q|_{h,g_{\nbigm}}=O\bigl(d(P,Q)^{-2}\bigr).
\end{gather*}
In particular,
$\bigl|\nabla_h\phi_h|_h$
and $\bigl|F(\nabla_h)\bigr|_h$ are $L^1$.
\end{Lemma}
\begin{proof}
Suppose that $\varrho=0$.
The estimates
$|\phi_{h,Q}|_h=O\bigl(d(P,Q)^{-1}\bigr)$
and
$|(\nabla\phi_h)_{Q}|_{h,g_{\nbigm}}=O\bigl(d(P,Q)^{-2}\bigr)$
directly follow from
\cite[Proposition~1]{Mochizuki-Yoshino}.
Because of Lemma~\ref{lem;19.10.20.2},
Lemma~\ref{lem;19.10.19.200}
and~(\ref{eq;19.10.20.3}),
we obtain
$|F(\nabla_h)_Q|_{h,g_{\nbigm}}=O\bigl(d(P,Q)^{-2}\bigr)$.
We can reduce the case $\varrho$
to the case $\varrho=0$
by using Lemma~\ref{lem;19.10.19.101}.
\end{proof}

\subsubsection[Analytic stability condition for $\varrho$-twisted mini-holomorphic bundles with Dirac type singularity]{Analytic stability condition for $\boldsymbol{\varrho}$-twisted mini-holomorphic bundles\\ with Dirac type singularity}

Let $\big(E,\delbar_E\big)$ be a $\varrho$-twisted
mini-holomorphic bundle
with Dirac type singularity on $(\nbigm;Z)$.
We set
\[
 \deg^{\an}\big(E,\delbar_E\big):=
 \deg\big(E,\delbar_E,h\big),
\qquad
 \mu^{\an}\big(E,\delbar_E\big):=
 \deg^{\an}\big(E,\delbar_E\big)/\rank(E)
\]
for an adapted Hermitian metric $h$ of $E$,
which is independent of the choice of $h$.
The numbers are called
the analytic degree and the analytic slope of
$\big(E,\delbar_E,h\big)$, respectively.

\begin{Definition}\label{df;19.10.18.10}
We say that $\big(E,\delbar_E\big)$ is analytically stable
if $\mu^{\an}\big(E',\delbar_{E'}\big)<\mu^{\an}\big(E,\delbar_E\big)$
holds for any $\varrho$-twisted mini-holomorphic subbundle
$E'\subset E$ with
$0<\rank(E')<\rank(E)$.
It is called polystable if
$\big(E,\delbar_E\big)=\bigoplus\big(E_i,\delbar_{E_i}\big)$,
where
each $\big(E_i,\delbar_{E_i}\big)$ is stable
such that
$\mu^{\an}\big(E_i,\delbar_{E_i}\big)
=\mu^{\an}\big(E,\delbar_{E}\big)$.
\end{Definition}

We obtain the following lemma from Lemma~\ref{lem;19.2.27.1}.
\begin{Lemma}\label{lem;19.8.10.40}
A $\varrho$-twisted
mini-holomorphic bundle with Dirac type singularity
$\big(E,\delbar_E\big)$ on $(\nbigm;Z)$
is analytically stable if and only if
$\big(E,\delbar_E,h\big)$ is analytically stable
for an adapted Hermitian metric~$h$ of~$E$.
\end{Lemma}

\subsubsection[Complement on the choice of $\varrho$]{Complement on the choice of $\boldsymbol{\varrho}$}

Let $H^i(\nbigm,\nbigo_{\nbigm})$ denote the $i$-th cohomology group
of the complex
$\bigl( C^{\infty}\big(\nbigm,\Omega^{0,i}\nbigm\big),\delbar_{\nbigm}
\bigr)$.
For any $\nu\in C^{\infty}\big(\nbigm,\Omega^{0,1}\nbigm\big)$,
$\varrho$-twisted mini-holomorphic bundles
are equivalent to
$(\varrho-\delbar\nu)$-twisted mini-holomorphic bundles.
Hence, the essential ambiguity of the choice of $\varrho$
lives in $H^2(\nbigm,\nbigo_{\nbigm})$.

\begin{Lemma}
We have the following isomorphisms:
\[
 H^0(\nbigm,\nbigo_{\nbigm})\simeq \cnum,
\qquad
 H^1(\nbigm,\nbigo_{\nbigm})\simeq
 \cnum\,dt\oplus \cnum\,d\wbar,
\qquad
 H^2(\nbigm,\nbigo_{\nbigm})\simeq
 \cnum\,dt\wedge d\wbar.
\]
Hence, for the study of twisted mini-holomorphic bundles,
it is essential to study the case
$\varrho=\alpha\,dt\,d\wbar$ for some $\alpha\in\cnum$.
\end{Lemma}

\begin{proof}
We have the isomorphism
$\real_s\times(\real_t\times\cnum_w)
\simeq\cnum_z\times\cnum_w$
given by
$(s,t,w)\longmapsto\big(s+\sqrt{-1}t,w\big)$.
We consider the action of $\seisuu\times\Gamma$
on $\real\times(\real\times\cnum)$
induced by the natural action of~$\seisuu$ on~$\real$
and the $\Gamma$-action on $\real\times\cnum$.
Let $X$ denote the quotient space.
We have the projection $\varphi\colon X\lrarr \nbigm$
induced by $(s,t,w)\longmapsto(t,w)$.
We have the natural $S^1=\real/\seisuu$-action on $X$,
and the quotient space is identified with $\nbigm$.
Let
$\varphi^{\ast}\colon C^{\infty}\big(\nbigm,\Omega^{0,i}\nbigm\big)
\lrarr
 C^{\infty}\big(X,\Omega^{0,i}(X)\big)$
be the map induced by
$\varphi^{\ast}(d\wbar)=d\wbar$,
$\varphi^{\ast}(dt)=\del_{\zbar}(t)\,d\zbar
=\frac{\sqrt{-1}}{2}d\zbar$
and
the natural pull back
$\varphi^{\ast}\colon C^{\infty}(\nbigm,\cnum)
\lrarr
 C^{\infty}(X,\cnum)$.
Then, it is easy to check that
it is a morphism of complexes,
and that it induces an isomorphism between
$C^{\infty}\big(\nbigm,\Omega^{0,\bullet}\nbigm\big)$
and
the $S^1$-invariant part of
$C^{\infty}\big(X,\Omega^{0,\bullet}(X)\big)$.
Therefore,
it induces the isomorphism of
$H^i(\nbigm,\nbigo_{\nbigm})$
and the $S^1$-invariant part of
$H^i(X,\nbigo_X)$.
Then, the claim of the lemma follows.
\end{proof}

\begin{Remark}
Let $\nbigm=\bigcup_{\lambda\in\Lambda} U_{\lambda}$
be an open covering such that
the following holds:
\begin{itemize}\itemsep=0pt
\item
There exist
 $\nu_{\lambda}\in
 C^{\infty}\big(U_{\lambda},\Omega^{0,1}_{U_{\lambda}}\big)$
such that $\varrho_{|U_{\lambda}}=\delbar\nu_{\lambda}$.
\item
There exist
 $\alpha_{\lambda,\mu}\in
 C^{\infty}(U_{\lambda}\cap U_{\mu})$
such that
 $\nu_{\lambda}-\nu_{\mu}=\delbar \alpha_{\lambda,\mu}$.
We assume that
$\alpha_{\lambda,\lambda}=0$
and
$\alpha_{\lambda,\mu}=-\alpha_{\mu,\lambda}$.
\end{itemize}
Let $\nbige_{\lambda}$
be the $\nbigo_{U_{\lambda}}$-module
obtained as the sheaf of
mini-holomorphic sections of
$\big(E_{U_{\lambda}},\delbar_E-\nu_{\lambda}\big)$.
We obtain the isomorphism
$\beta_{\lambda,\mu}\colon
 \nbige_{\lambda|U_{\lambda}\cap U_{\mu}}
\simeq
 \nbige_{\mu|U_{\lambda}\cap U_{\mu}}$
by the multiplication of
$\exp(-\alpha_{\lambda,\mu})$.
We obtain the holomorphic functions
$\theta_{\lambda,\mu,\kappa}$
on $U_{\lambda,\mu,\kappa}$
such that
$\beta_{\lambda,\mu}\circ\beta_{\mu,\kappa}\circ
 \beta_{\kappa,\lambda}
=\theta_{\lambda,\mu,\kappa}\id$.
Such a~tuple
$(\{\nbige_{\lambda}\},\{\beta_{\lambda,\mu}\})$
is called a twisted sheaf.
The cohomology class of
$[\theta_{\lambda,\mu,\kappa}]$
in $H^2(\nbigm,\nbigo_{\nbigm}^{\ast})$
depends only on $\varrho$,
and it is equal to the image of $\varrho$
via the natural map
$H^2(\nbigm,\nbigo_{\nbigm})\lrarr
 H^2(\nbigm,\nbigo^{\ast}_{\nbigm})$.
\end{Remark}

\subsection{Twisted difference modules and twisted mini-holomorphic bundles} \label{subsection;19.8.10.50}

We assume that
(i) the tuple $(a_i,\alpha_i)$ $(i=1,2,3)$ is an oriented base of
 $\real\times\cnum$,
(ii) $\alpha_1$ and $\alpha_2$ are linearly independent over $\real$,
(iii) the tuple $(\alpha_1,\alpha_2)$ is an oriented base of
 $\cnum$.
Let $\Gamma_0\subset\cnum$ be the lattice
generated by $\alpha_1$ and $\alpha_2$.

Let $\nbigm^{\cov}$ denote the quotient space of
$Y$ by the action of
$\seisuu\tte_1\oplus\seisuu\tte_2$.
We have the natural isomorphism
$\nbigm^{\cov}/\seisuu\tte_3\simeq\nbigm$.
The projection $Y\lrarr\cnum$
induces a morphism
$\nbigm^{\cov}\lrarr T:=\cnum/\Gamma_0$.

\subsubsection{Another mini-complex coordinate system}

\label{subsection;19.8.10.51}

We introduce another mini-complex coordinate system $(s,u)$ on $Y$.
We set
\[
 \gamma:=-\frac{a_1\alphabar_2-a_2\alphabar_1}
 {\alpha_1\alphabar_2-\alpha_2\alphabar_1}.
\]
We introduce another mini-complex coordinate system
$(s,u)$ on the mini-complex manifold $Y$ as follows:
\[
 s:=t+2\Re(\gamma w)=t+\gammabar\wbar+\gamma w,
\qquad
 u:=w.
\]
Then, we obtain
$\tte_i(s,u)=(s,u+\alpha_i)$ for $i=1,2$.
We also obtain
$\tte_3(s,u)=(s+\gminit,u+\gminia)$,
where
\[
 \gminit:=a_3+2\Re(\gamma\alpha_3),
\qquad
 \gminia:=\alpha_3.
\]
Note that $\gminit>0$,
which follows from that
the tuple
$\{(a_i,\alpha_i)\}_{i=1,2,3}$
is an oriented base of $\real\times\cnum$,
and that $\{\alpha_1,\alpha_2\}$
is an oriented base of $\cnum$.
We have the following relations
of complex vector fields:
\[
 \del_{\wbar}=\del_{\ubar}+\gammabar\del_s,
\qquad
 \del_w=\del_u+\gamma\del_s,
\qquad
 \del_t=\del_s.
\]

The product $\real_s\times T$ is equipped
with the natural mini-complex structure.
The mini-complex coordinate system $(s,u)$
induces an isomorphism of
mini-complex manifolds
$\nbigm^{\cov}\simeq
\real_s\times T$.

\subsubsection{Twisted mini-holomorphic bundles and twisted difference modules}\label{subsection;19.2.25.12}

Let $Z$ be a finite subset in $\nbigm$.
Let $Z^{\cov}\subset\nbigm^{\cov}\simeq\real_s\times T$
denote the pull back of $Z$.
For any $a<b$, we set
$\closedopen{a}{b}:=\{a\leq s<b\}$.
We take $\epsilon>0$
such that
$(\closedopen{-\epsilon}{0}\times T)\cap Z^{\cov}=\varnothing$.
Let $D$ be the image of
$Z^{\cov}\cap(\closedopen{-\epsilon}{\gminit}\times T)$
via the projection $\real_s\times T\lrarr T$.
For each $P\in D$,
we obtain the sequence
$0\leq s_{P,1}<s_{P,2}<\cdots<s_{P,m(P)}<\gminit$
by the condition:
\[
\{(s_{P,i},P)\,|\,i=1,\ldots,m(P)\}
=(\closedopen{0}{\gminit}\times\{P\})
 \cap
 Z^{\cov}.
\]
We set $\tau_{P,i}:=s_{P,i}/\gminit$.

We have the expression
$\varrho=\varrho_0\,dt\,d\wbar
=\varrho_0\,ds\,d\ubar$.
Let $\varrho_0^{\cov}$ be the function
on $\nbigm^{\cov}=\real_s\times T$
obtained as the pull back of $\varrho_0$
by $\nbigm^{\cov}\lrarr\nbigm$.
We define
$\nu_{\varrho}=\nu_{\varrho,\wbar}\,d\wbar
 \in C^{\infty}(\nbigm^{\cov},\Omega^{0,1}\nbigm)$
by setting
\[
 \nu_{\varrho,\wbar}(s,u)=\int_0^s\varrho^{\cov}_0(\sigma,u)\,d\sigma.
\]
We set
$\vartheta_{\varrho}:=\nu_{\varrho|\{\gminit\}\times T}$.
Let $\gbigl_{\varrho}$ be the holomorphic line bundle
on $T$
given by
the product bundle
$\cnum\times T$
with $\delbar_T-\vartheta_{\varrho}$.

Let $\big(E,\delbar_E\big)$ be a $\varrho$-twisted
mini-holomorphic bundle with
Dirac type singularity on $(\nbigm;Z)$.
Let us observe that
$(E,\delbar_E)$ induces a parabolic $\gbigl_{\varrho}$-twisted
difference module
$\Upsilon(E,\delbar_E)$
over $(T,(\vectau_{P})_{P\in D})$.

Let $\varrho^{\cov}\in C^{\infty}\big(\nbigm,\Omega^{0,1}\nbigm\big)$
be the pull back of $\varrho$.
Let $\big(E^{\cov},\delbar_{E^{\cov}}\big)$
denote the $\varrho^{\cov}$-twisted mini-holomorphic bundle
on $\nbigm^{\cov}$
obtained as the pull back of $\big(E,\delbar_E\big)$.
We set
$\big(\Etilde^{\cov},\delbar_{\Etilde^{\cov}}\big):=
 \big(E^{\cov},\delbar_{E^{\cov}}-\nu_{\varrho}\big)$
which is a mini-holomorphic bundle
on $\nbigm^{\cov}$.

Let $V$ be the locally free $\nbigo_T$-module
obtained as
$\Etilde^{\cov}_{|\{-\epsilon\}\times T}$.
It is independent of the choice of~$\epsilon$ as above,
up to canonical isomorphisms.

Let $\Phi\colon T\lrarr T$ be the morphism induced by
$\Phi(u)=u+\gminia$.
We have the natural isomorphism
\[
 \Phi^{\ast}\big(E^{\cov}_{|\{\gminit-\epsilon\}\times T}\big)
\simeq E^{\cov}_{|\{-\epsilon\}\times T}.
\]
It induces the following isomorphism
of holomorphic bundles on $T$:
\[
 \Phi^{\ast}
 \bigl(
 \big(\Etilde^{\cov},\delbar_{\Etilde^{\cov}}\big)_{|\{\gminit-\epsilon\}\times T}
 \bigr)
\simeq
 \big(\Etilde^{\cov},\delbar_{\Etilde^{\cov}}\big)_{|\{-\epsilon\}\times T}
 \otimes
 \gbigl_{\varrho}.
\]
The scattering map induces an isomorphism
\[
 \big(\Etilde^{\cov},\delbar_{\Etilde^{\cov}}\big)
 _{|\{-\epsilon\}\times T}(\ast D)
\simeq
 \big(\Etilde^{\cov},\delbar_{\Etilde^{\cov}}\big)
 _{|\{\gminit-\epsilon\}\times T}
 (\ast D).
\]
Hence, $V$ is equipped with an isomorphism
$V(\ast D)\simeq
 \bigl(
 (\Phi^{\ast})^{-1}(V)
 \otimes\gbigl_{\varrho}
\bigr) (\ast D)$.

For each $P\in D$ and for $i=1,\ldots,m(P)-1$,
we take $s_{P,i}<b_{P,i}<s_{P,i+1}$.
Let $\big(\Etilde^{\cov}_{|\{-\epsilon\}\times T}\big)_P$
denote the $\nbigo_{T,P}$-module obtained as
the stalk of the sheaf of holomorphic sections of
$\Etilde^{\cov}_{|\{-\epsilon\}\times T}$ at $P$.
Similarly,
$\big(\Etilde^{\cov}_{|\{b_{P,i}\}\times T}\big)_P$
denote the $\nbigo_{T,P}$-module obtained as
the stalk of the sheaf of holomorphic sections of
$\Etilde^{\cov}_{|\{b_{P,i}\}\times T}$ at $P$.
The scattering map induces isomorphisms of
$\nbigo_T(\ast P)_P$-modules:
\[
 \big(\Etilde^{\cov}_{|\{-\epsilon\}\times T}\big)_P(\ast P)
\simeq
 \big(\Etilde^{\cov}_{|\{b_{P,i}\}\times T}\big)_P(\ast P).
\]
Hence,
$\big(E^{\cov}_{|\{b_{P,i}\}\times T}\big)_P$
$(i=1,\ldots,m(P)-1)$
induce a sequence of lattices
$\nbigl_{P,i}$ $(i=1,\ldots,m(P)-1)$
of $V(\ast D)_{P}$.
Thus, we obtain
the following parabolic $\gminia$-difference
module on $(T,(\vectau_P)_{P\in D})$:
\[
 \Upsilon\big(E,\delbar_E\big):=
 \bigl(
 V,(\vectau_P,\vecnbigl_P)_{P\in D}
 \bigr).
\]
The following proposition is clear by the construction.
\begin{Proposition}\label{prop;19.3.2.10}
$\Upsilon$ induces
an equivalence between
$\varrho$-twisted mini-holomorphic bundles
with Dirac type singularity on $(\nbigm;Z)$
and
parabolic $\gbigl_{\varrho}$-twisted $\gminia$-difference modules
on $(T,(\vectau_P)_{P\in D})$.
\end{Proposition}

\subsubsection{Comparison of stability conditions}

Let $\big(E,\delbar_E\big)$ be a $\varrho$-mini-holomorphic bundle
with Dirac type singularity on $(\nbigm;Z)$.

\begin{Proposition}\label{prop;19.2.27.50}
We have
$\mu^{\an}\big(E,\delbar_E\big)
=\gminit\pi\mu\bigl( \Upsilon\big(E,\delbar_E\big)\bigr)
+2\int_{\nbigm}\Re(\gamma\varrho_0)$.
As a result,
$\big(E,\delbar_E\big)$ is analytically (poly)stable
if and only if
$\Upsilon\big(E,\delbar_E\big)$ is (poly)stable.
\end{Proposition}

\begin{proof} We consider the real vector field
$\gminiv:=
 2\gammabar\del_w
+2\gamma\del_{\wbar}
-\bigl(
 2|\gamma|^2-\frac{1}{2}
 \bigr)
 \del_t$
on $\nbigm$.
Let $h$ be any Hermitian metric of~$E$.
Let $\del_{E,\ubar}$ denote the operator on~$E$
induced by $\delbar_E$ and $\del_{\ubar}$.
Let $\del_{E,h,u}$ denote the operator on $E$
induced by $\del_{E,h}$ and $\del_{u}$.

\begin{Lemma}\label{lem;19.2.25.1}
$G(h)=
 \bigl[\del_{E,h,u},\del_{E,\ubar}\bigr]
 -\sqrt{-1}\nabla_{h,\gminiv}\phi_h
+2\Re(\gamma\varrho_0)\id_E$ holds.
\end{Lemma}
\begin{proof}
Because
$\del_{E,t}=\nabla_{h,t}-\sqrt{-1}\phi_h$
and
$\del_{E,h,t}=\nabla_{h,t}+\sqrt{-1}\phi_h$,
the following holds:
\[
 \del_{E,\ubar}
=\nabla_{h,\wbar}
-\gammabar\big(\nabla_{h,t}-\sqrt{-1}\phi_h\big),
\qquad
 \del_{E,h,u}
=\nabla_{h,w}
-\gamma\big(\nabla_{h,t}+\sqrt{-1}\phi_h\big).
\]
Hence, we obtain
\begin{gather*}
 \bigl[
 \del_{E,h,u},\del_{E,\ubar}
 \bigr]
=
\bigl[\nabla_{h,w},\nabla_{h,\wbar}
 \bigr]
-\gammabar\bigl[\nabla_{h,w},\nabla_{h,t}\bigr]
+\gammabar\sqrt{-1}\nabla_{h,w}\phi_h\\
\hphantom{\bigl[\del_{E,h,u},\del_{E,\ubar} \bigr]=}{}
+\gamma[\nabla_{h,\wbar},\nabla_{h,t}]
+\gamma\sqrt{-1}\nabla_{h,\wbar}\phi
-2\sqrt{-1}|\gamma|^2\nabla_{h,t}\phi_h.
\end{gather*}
According to Lemma \ref{lem;19.10.20.2},
we have
$\bigl[
 \nabla_{h,\wbar},\nabla_{h,t}
 \bigr]
-\sqrt{-1}\nabla_{h,\wbar}\phi
=-\varrho_0\id_E$
and
$\bigl[
 \nabla_{h,w},\nabla_{h,t}
 \bigr]
+\sqrt{-1}\nabla_{h,w}\phi
=\overline{\varrho_0}\id_E$.
Hence, we obtain
\begin{gather*}
 \bigl[
 \del_{E,h,u},\del_{E,\ubar}
 \bigr]
=\bigl[\nabla_{h,w},\nabla_{h,\wbar}\bigr]
+2\sqrt{-1}\gammabar\nabla_{w}\phi
+2\sqrt{-1}\gamma\nabla_{\wbar}\phi\nonumber\\
\hphantom{\bigl[ \del_{E,h,u},\del_{E,\ubar} \bigr]=}{}
-2\sqrt{-1}|\gamma|^2\nabla_t\phi
-2\Re(\gamma\varrho_0)\id_E.
\end{gather*}
Then, we obtain the claim of the lemma.
\end{proof}

Let $h$ be an adapted metric of $\big(E,\delbar_E\big)$.
According to Lemmas~\ref{lem;19.10.19.200} and~\ref{lem;19.10.18.30},
$G(h)$ and $\nabla_h\phi_h$ are $L^1$. Hence, we obtain
\begin{gather*}
 \deg^{\an}(E)
=\int_{\nbigm} \Tr G(h)
=\int_{\nbigm}
 \Tr \bigl[ \del_{E,h,u},\del_{E,\ubar}\bigr]\\
 \hphantom{\deg^{\an}(E)=}{}
-\int_{\nbigm}
 \sqrt{-1}\Tr\nabla_{h,\gminiv}\phi_h
+2\rank(E)\int_{\nbigm}\Re(\gamma\varrho_0).
\end{gather*}
Note that the volume form of $\nbigm$
is equal to
$\frac{\sqrt{-1}}{2}dt\,dw\,d\wbar
=\frac{\sqrt{-1}}{2}ds\,du\,d\ubar$.
By the Stokes theorem
and the estimate in Lemma~\ref{lem;19.10.18.30},
we obtain that
$\int_{\nbigm}\Tr\bigl(\nabla_{h,\gminiv}\phi_h\bigr)
 \frac{\sqrt{-1}}{2}dt\,dw\,d\wbar=0$.
By the Fubini theorem,
we obtain that
\begin{align*}
 \int_{\nbigm}
 \Tr\bigl[\del_{E,h,u},\del_{E,\ubar}\bigr]
& =\int_0^{\gminit}\,ds
 \int_{\{s\}\times T}
 \Tr\bigl[\del_{E,h,u},\del_{E,\ubar}\bigr]
 \,\frac{\sqrt{-1}}{2}\,du\,d\ubar
 \\
& =\int_0^{\gminit}\,ds
 \int_{\{s\}\times T}
 \pi c_1\big(E^{\cov}_{|\{s\}\times T}\big)
=\gminit\pi\deg\Upsilon\big(E,\delbar_E\big).
\end{align*}
Thus, we obtain the claim of Proposition~\ref{prop;19.2.27.50}.
\end{proof}

\subsection{Twisted monopoles and twisted mini-holomorphic bundles}\label{subsection;20.5.17.2}

\subsubsection{Statements}

Let $B$ be a real $2$-form on $\nbigm$.
We set $\varrho_B:=\sqrt{-1}B^{(0,1),\eta}$
and $\mu_B:=-\frac{1}{2}\int_{\nbigm}B^{(1,1)}\,dt$.
Let $(E,h,\nabla,\phi)$ be a $B$-twisted
monopole with Dirac type singularity
on $\nbigm\setminus Z$.
We have the associated
$\varrho_B$-twisted mini-holomorphic bundle
$\big(E,\delbar_E\big)$.
Note that
$G\big(h,\delbar_E\big)\,dw\,d\wbar=\sqrt{-1}B^{(1,1)}\id_E$.
Hence, we obtain
\[
 \mu^{\an}\big(E,\delbar_E\big)
=\frac{1}{\rank(E)}\int_{\nbigm}
 \Tr\big(G\big(h,\delbar_E\big)\big)
 \frac{\sqrt{-1}}{2}dt\,dw\,d\wbar
=-\frac{1}{2}
\int_{\nbigm}B^{(1,1)}\,dt
=\mu_B.
\]

We shall prove the following theorem
in Sections~\ref{subsection;19.3.3.1}--\ref{subsection;19.2.27.3},
which is a variant of
the correspondence in~\cite{Charbonneau-Hurtubise}
on the basis of~\cite{Simpson88}.

\begin{Theorem}\label{thm;19.2.27.10}
The above construction induces an equivalence
between
$B$-twisted monopoles with Dirac type singularity on
$\nbigm\setminus Z$
and analytically polystable
$\varrho_B$-twisted mini-holomorphic bundles
with Dirac type singularity with slope $\mu_B$
on $(\nbigm;Z)$.
\end{Theorem}

More precisely, Theorem \ref{thm;19.2.27.10} consists of Propositions~\ref{prop;19.3.5.1},~\ref{prop;19.3.5.3}, and~\ref{prop;19.3.5.2}
below.

\begin{Remark} According to Lemma~\ref{lem;19.10.19.30}
and Remark~\ref{rem;19.10.19.31},
it is essential to study the case where
\[
B=c\frac{\sqrt{-1}}{2}\,dw\wedge d\wbar
+\alpha\,dt\wedge d\wbar+
 \alphabar\,dt\wedge dw
\]
for $(c,\alpha)\in\real\times\cnum$.
We have
$\varrho_B=\sqrt{-1}\alpha\,dt\wedge dw$
and $\mu_B=-\frac{1}{2}\vol(\nbigm)c$
in this case.
\end{Remark}

\subsubsection{Polystability}\label{subsection;19.3.3.1}

Let $\big(E,\delbar_E,h\big)$ be a $B$-twisted
monopole with Dirac type singularity
on $\nbigm\setminus Z$.
\begin{Proposition}\label{prop;19.3.5.1}
$\big(E,\delbar_E\big)$ is analytically polystable
with $\deg^{\an}\big(E,\delbar_E\big)=\rank(E)\mu_B$.
\end{Proposition}
\begin{proof}
Let $E'$ be a $\varrho_B$-twisted
mini-holomorphic subbundle of $E$.
Let $h_{E'}$ be the metric of $E'$ induced by $h$.
By the Chern--Weil formula~(\ref{eq;19.2.27.2})
and Lemma~\ref{lem;19.2.27.1},
we have
\begin{gather*}
 \deg^{\an}\big(E',\delbar_{E'}\big)=
 \int \Tr G(h_{E'})=
\rank(E')\mu_B
-\int\bigl|\del_{E,\wbar}p_{E'}\bigr|^2\\
\hphantom{\deg^{\an}\big(E',\delbar_{E'}\big)=}{}
-\frac{1}{4}\int\bigl|\del_{E,t}p_{E'}\bigr|^2
\leq \rank(E')\mu_B.
\end{gather*}
If $\mu^{\an}\big(E',\delbar_{E'}\big)=\mu_B$,
we obtain
$\del_{E,\wbar}p_{E'}=\del_{E,t}p_{E'}=0$.
We obtain that
the orthogonal complement $E^{\prime\bot}$
is also a $\varrho_B$-twisted
mini-holomorphic subbundle of~$E$.
Let $h_{E^{\prime\bot}}$ be the metric of $E^{\prime\bot}$
induced by~$h$.
Thus, we obtain a decomposition of
monopoles
$\big(E,\delbar_E,h\big)=
\big(E',\delbar_{E'},h_{E'}\big)
\oplus
 \big(E^{\prime\bot},\delbar_{E^{\prime\bot}},h_{E^{\prime\bot}}\big)$.
Hence, we obtain the polystability of $\big(E,\delbar_E\big)$ by an easy induction.
\end{proof}

\subsubsection{Uniqueness}\label{subsection;19.3.3.2}

The uniqueness is also standard.
\begin{Proposition}\label{prop;19.3.5.3}
Let $\big(E,\delbar_E\big)$ be a $\varrho_B$-twisted
mini-holomorphic bundle
with Dirac type singularity on $(\nbigm;Z)$.
Let $h_1$ and $h_2$ be adapted Hermitian-metrics
of $E$
such that $G(h_i)\,dw\,d\wbar=\sqrt{-1}B^{(1,1)}\id_E$.
Then, there exists a decomposition
$\big(E,\delbar_E\big)=
 \bigoplus \big(E_j,\delbar_{E_j}\big)$
such that
$(i)$~it is orthogonal with respect to both~$h_1$ and~$h_2$,
$(ii)$~there exist positive constants~$a_j$
 such that
 $h_{2|E_j}=a_jh_{1|E_j}$.
\end{Proposition}
\begin{proof} Let $s$ be the automorphism of $E$ determined by $h_2=h_1s$.

\begin{Lemma}The following inequality holds on $\nbigm\setminus Z$:
\[
 -\left(
 \del_{\wbar}\del_w
 +\frac{1}{4}\del_{t}^2
 \right)
 \Tr(s)
=-\bigl|
 s^{-1/2}\del_{E,h_1,w}(s)
 \bigr|^2_{h_1}
-\frac{1}{4}
 \bigl|
 s^{-1/2}
 \del_{E,h_1,t}(s)
 \bigr|^2_{h_1}
\leq 0.
\]
\end{Lemma}
\begin{proof} In the case $\varrho_B=0$,
it follows from
\cite[Corollary 2.30]{Mochizuki-difference-modules}.
(Note that $\del_{E,h_1,t}$ is denoted by
$\del'_{E,h_1,t}$ in
\cite[Corollary 2.30]{Mochizuki-difference-modules}.)
Let us study the general case.
We have only to check the inequality locally around any point $P$
of $\nbigm\setminus Z$.
We take a small neighbourhood $U$ of $P$
and $\nu=\nu_t\,dt+\nu_{\wbar}\,d\wbar
 \in C^{\infty}(U,\Omega^{0,1})$
such that
$\delbar_E-\nu\,\id$
is mini-holomorphic.
We obtain
$\del^{\nu}_{E,h,w}=\del_{E,h,w}+\overline{\nu_{\wbar}}\id$
and
$\del^{\nu}_{E,h,t}=\del_{E,h,t}+\overline{\nu_t}\id$.
Hence, we obtain
$\del^{\nu}_{E,h,w}(s)
=[\del^{\nu}_{E,h,w},s]
=[\del_{E,h,w},s]
=\del_{E,h,w}(s)$.
Similarly,
we obtain
$\del^{\nu}_{E,h,t}(s)
=\del_{E,h,t}(s)$.
Hence, the general case
can be reduced to the case $\varrho_B=0$.
\end{proof}

By the assumption, $\Tr(s)\geq 0$ is bounded.
Then, the inequality holds on $\nbigm$ in the sense of distributions.
(See the proof of \cite[Proposition~2.2]{Simpson88}.)
Hence, we obtain that
$\Tr(s)$ is constant,
and
$\del_{E,h_1,w}(s)=\del_{E,h_1,t}(s)=0$.
Because $s$ is self-adjoint with respect to $h_1$,
we also obtain that
$\del_{E,\wbar}(s)=\del_{E,t}(s)=0$.
We obtain that the eigenvalues of $s$ are constant,
and the eigen decomposition
$E=\bigoplus E_i$ is compatible with
the mini-holomorphic structure.
Then, the claim of the proposition follows.
\end{proof}

\subsubsection{Construction of twisted monopoles}\label{subsection;19.2.27.3}

Let $\big(E,\delbar_E\big)$ be a stable $\varrho_B$-twisted
mini-holomorphic bundle
with Dirac type singularity on $(\nbigm;Z)$
with $\mu^{\an}\big(E,\delbar_E\big)=\mu_B$.

\begin{Proposition}\label{prop;19.3.5.2}
There exists a Hermitian metric $h$
of $\big(E,\delbar_E\big)$
such that $\big(E,\delbar_E,h\big)$
is a~$B$-twisted monopole
with Dirac type singularity on $\nbigm\setminus Z$.
\end{Proposition}
\begin{proof}As a preliminary,
let us consider the rank one case.
Note that the stability condition is trivial in the rank one case.

\begin{Lemma}\label{lem;19.2.28.20}
Assume $\rank E=1$.
Then, there exists a Hermitian metric $h$
of $\big(E,\delbar_E\big)$
such that $\big(E,\delbar_E,h\big)$ is a $B$-twisted monopole
with Dirac type singularity on $\nbigm\setminus Z$.
\end{Lemma}
\begin{proof}
We take a Hermitian metric $h_0$ of $E$
such that the following holds:
\begin{itemize}\itemsep=0pt
\item
Each $P\in Z$ has a neighbourhood $\nbigu_P$ in $\nbigm$
such that
(i) $G(h_0)=0$ on $\nbigu_P\setminus\{P\}$,
(ii) $P$ is Dirac type singularity of
 the monopole $\big(E,\delbar_{E},h_0\big)_{|\nbigu_P\setminus\{P\}}$.
\end{itemize}
Let $f$ be any $C^{\infty}$-function on $\nbigm$.
Note that
$G\big(h_0e^f\big)-G(h_0)=4^{-1}\Delta f$,
where $\Delta$ denote the Laplacian of $\nbigm$.
(See \cite[Section~2.8.4]{Mochizuki-difference-modules}
for the untwisted case.
The twisted case can be argued similarly.)
Because
\[
\int_{\nbigm} G(h_0)\frac{\sqrt{-1}}{2}\,dt\,dw\,d\wbar
=\mu_B=-\frac{1}{2}\int_{\nbigm}B^{(1,1)}\,dt,
\]
there exists an $\real$-valued $C^{\infty}$-function
$f_1$ such that
$(\Delta f_1)dw\,d\wbar=
-4\bigl(
 G(h_0)\,dw\,d\wbar
-\sqrt{-1}B^{(1,1)}
\bigr)$.
Then, the claim of
Lemma~\ref{lem;19.2.28.20}
follows.
\end{proof}

Let us study the case
where $\varrho_B=0$,
which implies $B=B^{(1,1)}$.
On $\real^4=\real\times \real^3$,
we use the real coordinate system $(s,t,x,y)$
and the complex coordinate system
$(z,w)$ given by $z=s+\sqrt{-1}t$ and $w=x+\sqrt{-1}y$.

Let $\Gammatilde$ denote the lattice of
$\real^4=\real\times(\real\times\cnum)$
generated by
$(1,0,0)$ and $(0,a_i,\alpha_i)$ $(i=1,2,3)$.
We consider the action of $\Gammatilde$ on
$\real^4$
induced by
the natural $\seisuu$-action on $\real$
and the $\Gamma$-action on $\real\times\cnum$.
Let $(X,g_X)$ denote the K\"ahler manifold
obtained as the quotient of $(\cnum^2,dz\,d\zbar+dw\,d\wbar)$
by the $\Gammatilde$-action.
Let $p\colon X\lrarr \nbigm$ denote the naturally defined projection.

We set $\Etilde:=p^{-1}(E)$ on $X\setminus p^{-1}(Z)$.
It is equipped with the complex structure
$\delbar_{\Etilde}$
determined by
\[
\del_{\Etilde,\wbar}p^{-1}(u)=
 p^{-1}(\del_{E,\wbar}u),
\qquad
 \del_{\Etilde,\zbar}p^{-1}(u)=
 \frac{1}{2}\cdot
 p^{-1}\big( \phi\cdot u+\sqrt{-1}\del_{E,t}u\big)
\]
for sections $u$ of $E$.
For any adapted Hermitian metric $h_0$ of $E$,
set $\htilde_0:=p^{-1}(h_0)$.

Let $F\big(\htilde_0\big)$
denote the curvature of the Chern connection of
$\big(\Etilde,\delbar_{\Etilde},\htilde_0\big)$.
Let $\Lambda$ denote the contraction
from $(1,1)$-forms to $(0,0)$-forms
with respect to the K\"ahler form of~$(X,g_X)$.
Then,
$\sqrt{-1}\Lambda F\big(\htilde_0\big)=p^{-1}\bigl(G(h_0)\bigr)$
holds.

For any saturated coherent
$\nbigo_{X\setminus p^{-1}(Z)}$-submodule
$\Etilde'\subset \Etilde$,
we have a closed complex analytic subset
$W\subset X\setminus p^{-1}(Z)$
with complex codimension $2$
such that $\Etilde'$ is a subbundle of $\Etilde$
outside of $W$.
We have the induced metric $\htilde_{0,\Etilde'}$
of $\Etilde'_{|X\setminus (p^{-1}(Z)\cup W)}$.
We define
\[
 \deg\big(\Etilde',\htilde_0\big):=
 \sqrt{-1}\int \Tr\Lambda F\big(\htilde_{0,\Etilde'}\big)\dvol_X.
\]
Because of the Chern--Weil formula,
it is well defined in $\real\cup\{-\infty\}$
as explained in~\cite{Simpson88}.
Then,
$\big(\Etilde,\delbar_{\Etilde},\htilde_0\big)$
is defined to be analytically stable
with respect to the $S^1$-action
if
\[
 \frac{\deg\big(\Etilde',\htilde_0\big)}{\rank \Etilde'}
<
 \frac{\deg\big(\Etilde,\htilde_0\big)}{\rank \Etilde}
\]
holds for any $S^1$-invariant
saturated subsheaf $\Etilde'\subset \Etilde$
with
$0<\rank\Etilde'<\rank \Etilde$.
The following is clear.
\begin{Lemma}\label{lem;19.2.27.4}
$\big(\Etilde,\delbar_{\Etilde},\htilde_0\big)$
is analytically stable with respect to the $S^1$-action
if and only if
$\big(E,\delbar_E,h_0\big)$ is analytically stable.
\end{Lemma}

According to Lemma~\ref{lem;19.2.28.20},
there exists a Hermitian metric $h_{\det(E)}$
such that
$\big(E,\delbar_E,h_{\det(E)}\big)$
is a $(\rank E)B$-twisted monopole.
We take an adapted Hermitian metric $h_{-1}$
such that
each $P\in Z$ has a neighbourhood $\nbigu_P$
such that
$G(h_{-1})_{|\nbigu_P\setminus\{P\}}=0$.
An $\real$-valued $C^{\infty}$-function $f$
is determined by
$\det(h_{-1})=h_{\det(E)}e^{f}$.
We set $h_0=h_{-1}e^{-f/\rank(E)}$.
Then, $h_0$ is an adapted metric of $E$.
By Lemma~\ref{lem;19.2.27.4},
$\big(\Etilde,\delbar_{\Etilde},\htilde_0\big)$
is analytically stable with respect to
the $S^1$-action.
We also have
$\Lambda \Tr F\big(\htilde_0\big)=\sqrt{-1}\rank(E)p^{-1}(B)$.
According to a theorem of Simpson
\cite[Theorem~1]{Simpson88},
there exists an $S^1$-invariant metric
$\htilde$ of $\Etilde$
such that
(i) $\det\big(\htilde\big)=\det\big(\htilde_0\big)$,
(ii) $\Lambda F\big(\htilde\big)=\sqrt{-1}p^{-1}(B)\id_{\Etilde}$,
(iii) $\htilde$ and $\htilde_0$ are mutually bounded.
We obtain the corresponding metric~$h$ of~$E$,
for which $G(h)=\sqrt{-1}B\id_E$ holds.
Because~$h$ and~$h_0$ are mutually bounded,
each $P\in Z$ is a Dirac type singularity of
$\big(E,\delbar_E,h\big)$
which is implied by \cite[Theorem~3]{Mochizuki-Yoshino}.
Thus, we obtain the claim of Proposition~\ref{prop;19.3.5.2}
in the case $\varrho_B=0$.

Let us study the case where $\varrho_B$ is not necessarily $0$.
\begin{Lemma}
There exist a finite subset $Z_1\subset\nbigm$
and a $\varrho_B$-twisted mini-holomorphic bundle
$\big(E_1,\delbar_{E_1}\big)$
with Dirac type singularity of rank one
on $(\nbigm;Z_1)$
such that
$\deg^{\an}\big(E_1,\delbar_{E_1}\big)=\mu_B$.
\end{Lemma}
\begin{proof}
It follows from Lemma \ref{lem;19.10.19.300} and Proposition \ref{prop;19.3.2.10}.
\end{proof}

We set
$\big(E',\delbar_{E'}\big):=\big(E,\delbar_E\big)\otimes\big(E_1,\delbar_{E_1}\big)^{-1}$.
Then,
$\big(E',\delbar_{E'}\big)$ is a stable mini-holomorphic bundle
with $\mu^{\an}\big(E',\delbar_{E'}\big)=0$.
According to the claim in the case $\varrho_B=0$,
there exists an adapted Hermitian metric $h'$
of $\big(E',\delbar_{E'}\big)$
such that
$\big(E',\delbar_{E'},h'\big)$ is a monopole.
According to Lemma~\ref{lem;19.2.28.20},
there exists a Hermitian metric $h_1$ of $E_1$
such that
$\big(E_1,\delbar_{E_1},h_1\big)$
is a $B$-twisted monopole
with Dirac type singularity.
Let $h$ be the Hermitian metric of $E$
induced by~$h'$ and~$h_1$.
Then, $h$ is adapted to $\big(E,\delbar_E\big)$,
and
$\big(E,\delbar_E,h\big)$ is a $B$-twisted monopole.
Thus the proof of Proposition~\ref{prop;19.3.5.2}
is completed.
\end{proof}

\section{A more sophisticated formulation of the stability condition}

We explain that the analytic stability condition (Definition~\ref{df;19.10.18.10})
is equivalent to the stability condition
introduced by Kontsevich and Soibelman
in the case $\varrho=0$ (see Section~\ref{subsection;19.3.2.20}).
This section is devoted to explain their idea of degree.

\subsection{Preliminary}

\subsubsection{Closed 1-forms and 1-homology }

Let $A$ be a $3$-dimensional manifold.
Let $Z^i_{\DR}(A)$ denote the space of
closed $i$-form $\tau$ on $A$.
Let $B$ be finite subset of $A$.
Let $H_j(A,B)$ denote the relative $j$-th homology group
with $\real$-coefficient.

Let $\gamma$ be any element of $H_1(A,B)$.
We take a representative of $\gamma$
by a smooth $1$-chain $\gammatilde$.
For any $\omega\in Z^1_{\DR}(A)$,
the number $\int_{\gammatilde}\omega$
is independent of the choice of
a representative $\gammatilde$.
They are denoted by $\int_{\gamma}\omega$.

Let $C^{\infty}(A,B)$ denote the space of
$C^{\infty}$-functions $f$ on $A$
such that $f(P)=0$ for any $P\in B$.
Let $Z^1_{\DR}(A)$ denote the space of
closed $1$-forms on $A$.
Let $B^1_{\DR}(A,B)$ denote the image of
$d\colon C^{\infty}(A,B)\lrarr Z^1_{\DR}(A)$.
Because $\int_{\gamma}df=0$
for any $f\in C^{\infty}(A,B)$,
we obtain the well defined map
\[
 \int_{\gamma}\colon \
 Z^1_{\DR}(A)\big/
 B^1_{\DR}(A,B)
\lrarr\real.
\]

\subsubsection{Duality}

Suppose that $A$ is compact and oriented.
Let $H^j(A\setminus B)$
denote the $j$-th de Rham cohomology group
of $A\setminus B$.
Let $H^j_c(A\setminus B)$
denote the $j$-th de Rham cohomology group
with compact support.
We have the non-degenerate pairing between
$H^2(A\setminus B)$
and $H^1_c(A\setminus B)$
induced by the cup product and the integration.
We also have the non-degenerate pairing
between
$H^1_c(A\setminus B)$
and
$H_1(A,B)$
induced by the integration.
Hence, we obtain the isomorphism
\[
 \Phi_{A,B}\colon \ H^2(A\setminus B)\simeq H_1(A,B).
\]
By definition,
for any $a\in H^2(A\setminus B)$
and $b\in H^1_c(A\setminus B)$,
the following holds:
\[
 \int_{\Phi_{A,B}(a)}b=\int_{A}a\wedge b.
\]

Take any Riemannian metric $g_A$ of $A$.
For any $j$-form $\tau$ on $A\setminus B$,
let $|\tau|_{g_A}$ denote the function on $A\setminus B$
obtained as the norm of $\tau$ with respect to $g_A$.
\begin{Lemma}\label{lem;19.2.27.12}
Let $\tau\in Z^2_{\DR}(A\setminus B)$
such that $|\tau|_{g_A}$ is an $L^1$-function on $A$.
Then, the following holds
for any $\rho\in Z^1_{\DR}(A)$:
\[
 \int_{\Phi_{A,B}([\tau])}\rho
=\int_{A}\rho\wedge \tau.
\]
Here, $[\tau]\in H^2_{\DR}(A\setminus B)$
denotes the cohomology class of $\tau$.
\end{Lemma}

\begin{proof} For any point $P\in Z$,
we take a small coordinate neighbourhood
$(A_P,x_{P,1},x_{P,2},x_{P,3})$ of $P$ such that
(i) $P$ corresponds to $(0,0,0)$,
(ii) $A_P\simeq \big\{(x_1,x_2,x_3)\in\real^3\,|\,\sum x_i^2<1\big\}$
by the coordinate system.
Set $\|\vecx_P\|:=\bigl(x_{P,1}^2+x_{P,2}^2+x_{P,3}^2\bigr)^{1/2}$.
Then, there exists a $C^{\infty}$-function $f_P$
on $A_P$ such that
(i) $df_P=\rho$ on $\{\|\vecx_P\|<1/2\}$,
(ii) $f_P(P)=0$,
(iii) $f_P(Q)=0$ for $Q\in \{\|\vecx_P\|>2/3 \}$.
We naturally regard $f_P$ as a $C^{\infty}$-function on $A$.
Then, the following holds:
\begin{align*}
\int_{\Phi_{A,B}([\tau])}
 \rho
& =\int_{\Phi_{A,B}([\tau])}
 \bigg(
 \rho-\sum_{P\in B}df_P
 \bigg)\\
 & =\int_{A}
 \bigg(
 \rho-\sum_{P\in B}df_P
 \bigg)\wedge\tau
=\int_A\rho\wedge\tau
-\sum_P
 \int_Ad(f_P\tau).
\end{align*}
For each $P$, we set $S_P^2(r):=\bigl\{\|\vecx_P\|=r\bigr\}$
with the orientation as the boundary of
$\bigl\{\|\vecx_P\|\leq r\bigr\}$.
Then, we obtain the following
\begin{gather}\label{eq;19.3.1.1}
 \int_Ad(f_P\tau)
=-\lim_{\epsilon\to 0}\int_{S_P^2(\epsilon)}f_P\tau.
\end{gather}
Note that the limit exists
because $d(f_P\tau)=df_P\wedge \tau$ is integrable.
Because $|\tau|_{g_A}$ is $L^1$,
we have
$\int dr
 \int_{S_P^2(r)}|\tau|_{g_A}<\infty$,
and hence there exists a sequence
$r_i\to 0$ such that
$r_i\int_{S_P^2(r_i)}|\tau|_{g_A}\to 0$.
Because $|f_P|=O(\|\vecx_P\|)$,
we obtain
that (\ref{eq;19.3.1.1}) is $0$.
\end{proof}

\subsection{Relation between degrees of mini-holomorphic bundles}\label{subsection;19.8.10.30}

Let $\nbigm$ be as in Section~\ref{section;19.10.19.400}.
We may naturally regard $\nbigm$
as a $3$-dimensional abelian Lie group.
Let $\gbigt$ denote the space of the invariant
vector fields on $\nbigm$.
Let $\gbigt^{\lor}$ denote the space of
the invariant $1$-forms on $\nbigm$.
We have the natural non-degenerate paring
$\gbigt\otimes\gbigt^{\lor}\lrarr\real$.
We have the dual morphism
$\real\lrarr\gbigt^{\lor}\otimes\gbigt$.
Let $\sigma$ denote the image of $1$.
If we take a base $e_i$ $(i=1,2,3)$ of $\gbigt$
and the dual frame $e_i^{\lor}$ $(i=1,2,3)$,
then $\sigma=\sum e_i^{\lor}\otimes e_i$.
For the mini-complex coordinate
$(t,w)$,
we have
$\sigma=dt\otimes\del_t+dw\otimes\del_w+d\wbar\otimes\del_{\wbar}$.

Let $E$ be a vector bundle on $\nbigm\setminus Z$.
Kontsevich and Soibelman \cite{Kontsevich-Soibelman}
introduced the following element:
\[
\int_{\Phi_{\nbigm,Z}(c_1(E))}
 \sigma
\in \gbigt.
\]

\begin{Proposition}\label{prop;19.2.27.21} Let $\varrho=\varrho_0\,dt\,d\wbar$ be a $2$-form on $\nbigm$.
Let $\big(E,\delbar_E\big)$ be a $\varrho$-twisted mini-holomorphic bundle with Dirac type singularity on $(\nbigm;Z)$. Then,
\[
 \int_{\Phi_{\nbigm,Z}(c_1(E))}\sigma
=\frac{1}{\pi}\deg^{\an}(E)\cdot\del_t
-\frac{\rank(E)}{\pi}
\left(
 \bigg(
 \int_{\nbigm}\varrho_0
 \bigg)\del_{w}
+ \bigg(
 \int_{\nbigm}\overline{\varrho_0}
 \bigg)\del_{\wbar}
\right).
\]
In particular,
if $\varrho=0$,
then the following holds:
\[
 \int_{\Phi_{\nbigm,Z}(c_1(E))}\sigma
=\frac{1}{\pi}\deg^{\an}(E)\cdot\del_t.
\]
\end{Proposition}
\begin{proof}
Let $h$ be an adapted metric of $\big(E,\delbar_E\big)$.
By Lemma \ref{lem;19.2.27.12},
it is enough to prove the following equality:
\begin{gather}
 \frac{\sqrt{-1}}{2}
 \int_{\nbigm}\Tr F(h)\cdot\sigma
=\int_{\nbigm}\Tr G(h)\dvol_{\nbigm}\cdot \del_t\nonumber\\
 \hphantom{\frac{\sqrt{-1}}{2} \int_{\nbigm}\Tr F(h)\cdot\sigma=}{}
-\rank(E)\left(
 \bigg(
 \int_{\nbigm}\varrho_0
 \bigg)\del_{w}
+ \bigg(
 \int_{\nbigm}\overline{\varrho_0}
 \bigg)\del_{\wbar}
\right).\label{eq;18.1.25.1}
\end{gather}
For $\kappa=t,w,\wbar$,
we obtain the following by the Stokes formula
and the estimate $|\phi_{h,Q}|_h=O\bigl(d(P,Q)^{-1}\bigr)$:
\begin{gather}\label{eq;19.3.3.10}
 \int
 \Tr\bigl(
 \nabla_{h,\kappa}\phi_h
 \bigr)\frac{\sqrt{-1}}{2}dt\,dw\,d\wbar
=0.
\end{gather}

Note that
$F(h)_{t\wbar}+\sqrt{-1}\nabla_{\wbar}\phi
=\varrho_0\id_E$
and
$F(h)_{tw}-\sqrt{-1}\nabla_{w}\phi
=-\overline{\varrho_0}\id_E$,
according to Lemma~\ref{lem;19.10.20.2}.
We obtain
\begin{align*}
 \frac{\sqrt{-1}}{2}
 \int \Tr F(h)\,dw\otimes \del_w
& =\int \Tr \bigl(
F(h)_{t\wbar}
+\sqrt{-1}\nabla_{\wbar}\phi
\bigr)\frac{\sqrt{-1}}{2}dt\,d\wbar\,dw
 \otimes\del_w
 \\
& =-\rank(E)
\bigg(\int_{\nbigm}\varrho_0
 \bigg)\del_w.
\end{align*}
Similarly, we obtain
\[
 \frac{\sqrt{-1}}{2}
 \int \Tr F(h)\,d\wbar\otimes \del_{\wbar}
=-\rank(E)
\bigg(\int_{\nbigm}\overline{\varrho_0}
 \bigg)\del_{\wbar}.
\]

We also obtain the following from (\ref{eq;19.3.3.10}):
\begin{align*}
 \frac{\sqrt{-1}}{2}
 \int\Tr F(h)_{w\wbar}\,dw\,d\wbar\,dt\otimes\del_t
& =\int
 \Tr
 \left(
F(h)_{w\wbar}
 -\frac{\sqrt{-1}}{2}\nabla_{h,t}\phi_h
 \right)
 \frac{\sqrt{-1}}{2}\,dw\,d\wbar\,dt\,\otimes\del_t
\\
 & =\bigg(
 \int_{\nbigm} \Tr G(h)
 \bigg)\del_t.
\end{align*}
Thus, we obtain~(\ref{eq;18.1.25.1}), and the proof of Proposition \ref{prop;19.2.27.21} is completed.
\end{proof}

\begin{Remark}As explained in Section~\ref{subsection;19.3.2.20}, Kontsevich and Soibelman~\cite{Kontsevich-Soibelman} formulated the stability condition for mini-holomorphic bundles in terms of the coefficient of $\del_t$ in $\int_{\Phi_Z(c_1(E))}\sigma$.
\end{Remark}

\subsection*{Acknowledgements}

I thanks Maxim Kontsevich and Yan Soibelman for the communications and for sending the preprint~\cite{Kontsevich-Soibelman}. Indeed, this study grew out of my answer to one of their questions.
They also kindly suggested that there should be a generalization to the twisted case. I hope that this would be useful for their big project. I owe much to Carlos Simpson whose ideas on the Kobayashi--Hitchin correspondence are fundamental in this study. I thank Masaki Yoshino for discussions. I thank the referees for their careful readings and valuable comments.

I am grateful to the organizers of the conference ``Integrability, Geometry and Moduli'' to ce\-lebrate 60th birthday of Motohico Mulase. The twisted version of the equivalences was explained in my talk at the conference.

It is my great pleasure to dedicate this paper to Motohico Mulase with appreciation to his friendly encouragements and suggestions on many occasions.

I am partially supported by
the Grant-in-Aid for Scientific Research~(S) (No.~17H06127),
the Grant-in-Aid for Scientific Research~(S) (No.~16H06335),
and the Grant-in-Aid for Scientific Research~(C) (No.~15K04843),
Grant-in-Aid for Scientific Research~(C) (No.~20K03609),
Japan Society for the Promotion of Science.

\pdfbookmark[1]{References}{ref}
\LastPageEnding

\end{document}